%% file: stochastic-submodular_MS_camera_ready.tex
\documentclass[mnsc]{informs3}

\usepackage{latexsym, color, xcolor, colortbl,  amsmath, amssymb, amsfonts, enumerate, hyperref,dsfont} %

\usepackage{natbib}
 \bibpunct[, ]{(}{)}{,}{a}{}{,}%
 \def\bibfont{\small}%
 \def\newblock{\ }%

\OneAndAHalfSpacedXI

\def\todo{\textcolor{green}{TODO}}

\newcommand{\upt}{}

\EquationsNumberedThrough    

\MANUSCRIPTNO{MS-14-00784-R1}

\newtheorem{theorem}{Theorem}
\newtheorem{lemma}{Lemma}

\newtheorem{definition}{Definition}
\newtheorem{proposition}{Proposition}

\renewcommand{\qed}{\hfill \mbox{\raggedright $\square$}}

\begin{document}

%
\RUNAUTHOR{Asadpour and Nazerzadeh}

\RUNTITLE{Maximizing Stochastic Monotone Submodular Functions}

\TITLE{{\Large Maximizing Stochastic Monotone Submodular Functions}}

\ARTICLEAUTHORS{
\AUTHOR{Arash Asadpour}
\AFF{Stern School of Business, New York University, New York, NY, \EMAIL{aasadpou@stern.nyu.edu}}
\AUTHOR{Hamid Nazerzadeh}
\AFF{Marshall School of Business, University of Southern California, Los Angeles, CA, \EMAIL{hamidnz@marshall.usc.edu}} 
}
\AREAOFREVIEW{Optimization.}
\SUBJECTCLASS{Analysis of algorithms: performance guarantee; Nonlinear programming: submodular maximization; Programming: stochastic optimization.}

\def\newblock{\hskip .11em plus .33em minus .07em}
\def\ff{{f}}
\def\FF{{\mathcal F}}
\def\1{{\mathds{1}}}
\def\D{}
\def\c{{\textbf{c}}}
\def\A{{\cal A}}
\def\B{{\cal B}}
\def\F{{\cal A}}
\def\I{{\cal I}}
\def\M{{\cal M}}
\def\P{{\cal P}}
\def\V{{\cal V}}
\def\ih{x_i}
\def\Ih{\hat{I}}
\def\yh{\hat{y}}
\def\sh{\hat{S}}
\def\R{{\mathbb R}}
\def\E{{\textrm{\textbf{E}}}}
\def\Pr{{\textrm{\textbf{Pr}}}}
\def\Var{{\textrm{\textbf{Var}}}}
\def\Cov{{\textrm{Cov}}}
\def\efac{1-\frac{1}{e}}
\def\kfac{(1-\frac{1}{k})}
\def\ith{i^{\tiny{\mbox{th}}}}
\def\cte{{1\over e}}
\def\C{{\mathcal C}}
\def\adp{{\sc adapt}}
\def\todo{{\bf TODO\ }}
\def\argmax{\mbox{argmax}}
\def\Rnonneg{{\R_+}}
\def\nill{{\circ}}
\def\Thetazir{{\underline{\Theta}}}
\def\thetazir{{\underline{\theta}}}
\def\OPTNA{\textrm{OPT}_{\textrm{NA}}}
\def\OPTA{\textrm{OPT}_{\textrm{A}}}
\def\OPT{\textrm{OPT}}
\def\bary{\bar{y}}
\def\haty{\hat{y}}
\def\S{{\mathcal X}}

\ABSTRACT{
We study the problem of maximizing a stochastic monotone submodular function with respect to a matroid constraint. Due to the presence of diminishing marginal values in real-world problems, our model can capture the effect of stochasticity in a wide range of applications. We show that the adaptivity gap --- the ratio between the values of optimal adaptive and optimal non-adaptive policies --- is bounded and is equal to $\tfrac{e}{e - 1}$. We propose a polynomial-time non-adaptive policy that achieves this bound. 
We also present an adaptive myopic policy that obtains at least half of the optimal value. Furthermore, when the matroid is uniform,
the myopic policy achieves the optimal approximation ratio of $\efac$.
}

\maketitle

\input{introduction}

\input{definition}

\input{gap}

\input{matroidpolicies}

\input{discussion}

%
\input{Conclusion}

\medskip

\ACKNOWLEDGMENT{We would like to thank Amin Saberi, Jan Vondr{\'a}k,  Marek Adamczyk, Chandra Chekuri, and the anonymous referees whose comments and suggestions significantly improved this work.
We acknowledge that the example is section~\ref{sec:tightexample} was brought to our attention by Jan Vondr{\'a}k.}

\medskip

\begin{APPENDICES}
\input{appendix}

%

\bibliographystyle{ormsv080}
\bibliography{references}

\newpage

%
\input{online_appendix}

\end{APPENDICES}



\end{document}

%% file: introduction.tex
\section{Introduction} \label{sec:intro}
The problem of maximizing submodular  functions has been extensively
studied in operations research and computer science. For a ground set $\A$,
the set function $f: 2^{\A} \rightarrow \R$ is called \emph{submodular} if for any
two subsets $S, T \subseteq \A$, we have $f(S \cup T) + f(S \cap T) \le f(S) + f(T)$.
An equivalent definition is that the marginal value of adding any element is diminishing. In other words, for any $S\subseteq T\subseteq \A$ and $j\in \A$, we have $f(T\cup\{j\}) - f(T) \le f(S\cup\{j\}) - f(S).$
Also, a set function $f$ is called \emph{monotone} if for any two subsets $S \subseteq T \subseteq \A$, we have $f(S) \le f(T).$

Due to the presence of diminishing marginal values, a wide range of optimization problems that arise in the real world can be modeled as maximizing a monotone submodular function with
respect to some feasibility constraints. One instance is the welfare
maximization problem (see, e.g.,~\citealt{DobzinskiS06, Vondrak08, Feige06}), which is to find an optimal allocation of resources to agents where the
utilities of the agents are submodular. Here, submodularity corresponds to
the law of diminishing return in the economy (see~\citealt{Samuelson01}).

Another application of this problem is capital budgeting in which a
risk-averse investor with a limited budget is interested in finding
the optimal investment portfolio over different projects (see, e.g.,~\citealt{Weingartner63,AhmedA08}). Risk aversion can be modeled by concave utility functions where, informally speaking, the utility of a certain expectation is higher than the expectation of the utilities of the corresponding uncertain outcomes. As argued by~\cite{AhmedA08}, the utility function of a
risk-averse investor over a set of investment possibilities can be modeled by submodular functions. Such functions are also non-negative and monotone by nature.

Another example is the problem of viral marketing and maximizing
influence through the network (see, e.g.,~\citealt{KempeKT03,MosselR07}), where the goal is to choose an initial ``active'' set of people and to provide them with free coupons or promotions so as to
maximize the spread of a technology or behavior in a social network.
It is well known that under a wide variety of models for influence propagation in
networks (e.g.,~cascade model of~\citealt{KempeKT03}), the expected size
of the final cascade is a submodular function of the set of
initially activated individuals. Also, due to budget limitations,
the number of people that we can directly target is
bounded. Hence, the problem of maximizing influence can be seen as a
maximizing submodular function problem subject to cardinality
constraints.

Yet another example is the problem of optimal  placement of sensors
for environmental monitoring (see, e.g.,~\citealt{KrauseG05,KrauseG07}) where the objective is to place sensors in the environment in order to most
effectively reduce uncertainty in observations. This problem can be
modeled by entropy minimization and, because of concavity of the
entropy function, it is a special case of submodular optimization. 

For the above problems and many others, the constraints can be modeled by a matroid. A finite {\em matroid} $\M$ is defined by a pair $(\A, \I)$, where $\A$ is a ground set of elements and $\I$ is a collection of subsets of $\A$ (called the {\em independent sets}) with the following properties:
\begin{enumerate}
\item  Every subset of an independent set is independent, i.e., if $T \in \I$ and $S \subseteq T$, then $S \in \I$.
\item  If $S$ and $T$ are two independent sets and $T$ has more elements than $S$, then there exists an element in $T$ that is not in $S$ and when added to $S$ still gives an independent set.
\end{enumerate}

Two important special cases are {\em uniform matroid} and {\em partition matroid}.
In a uniform matroid, all the subsets of $\A$ of size at most $k$,
for a given $k$, are independent. Uniform matroids represent cardinality constraints.
A partition matroid is defined over a partition of set $\A$, where every independent set includes at most one element
from each set in the partition.

The celebrated result of~\cite{NemhauserWF78} shows
that for maximizing nonnegative monotone submodular functions over uniform matroids, the greedy algorithm  gives a $(\efac \approx 0.632)$-approximation of the optimal solution.
Later, \cite{FisherNW78-2} showed that for optimizing over matroids, the approximation ratio of the greedy algorithm is ${1\over 2}$.
Recently,~\cite{CalinescuCPV08} proposed a better approximation algorithm with a ratio of $1-{1\over e}$. It also has been shown that this factor is optimal (in the value oracle model, where we only have access to the values of $f(S)$ for all the subsets $S$ of the ground set), if only a polynomial
number of queries is allowed (see~\citealt{NemhauserW78,Feige98}).

However, all these algorithms are designed for deterministic environments. In practice, one must deal with the stochasticity caused by the uncertain nature of the problem's parameters, the incomplete information about the environment, and so on. For instance, in welfare maximization, the quality of the resources may be unknown in advance, or in the capital budgeting problem some projects taken by an investor may fail due to unexpected events in the market. 
Yet another example is from viral marketing, where some people whom we target might not adopt the behavior (e.g., receive a coupon or a promotion but not purchase the product).
Also, in the environmental monitoring example, it is expected that a non-negligible fraction of sensors might not be functioning properly due to various reasons.

All these possibilities motivate the study of submodular maximization in the stochastic setting. In such settings, the outcome of the elements in the selected set are not known in advance and they will only be discovered after they are chosen.

\subsection*{Setting}
{Here, we provide a brief overview of our setting. Later in Section~\ref{sec:def}, we will discuss the model in detail. We study the problem of maximizing a monotone submodular function $f$ over a set of $n$ random variables, namely, $\A = \{X_1, X_2, \cdots, X_n\}$. 
A policy $\pi$ picks the elements one by one (perhaps, based on the realized value of the previous elements) until it stops. Once $\pi$ stops, the state of the world is a random vector $\Theta^{\pi} = (\theta_1, \theta_2, \cdots, \theta_n)$, where $\theta_j$ denotes the realization of $X_i$, if $i$ is picked by the policy, and is equal to $0$ otherwise.
The objective of stochastic submodular maximization would be to optimize the expected value of a policy, i.e., $\underset{\pi}{\textrm{Maximize}}~\E[f(\Theta^{\pi})]$, subject to feasibility.
We model the feasibility constraint using a matroid. For a given matroid $\M$ defined on the ground set of the aforementioned random variable set $\A$, a policy $\pi$ is called feasible if the subset of random variables it picks is always an independent set of $\M$.}
%
%
%

We consider both \textbf{adaptive} (general) and \textbf{non-adaptive} policies. In adaptive policies, at each point in time all the information regarding the previous actions of the policy is known. In other words, the policy has access to the actual realized value of all the elements it has picked so far. In contrast, non-adaptive policies do not have access to such information and should make their decisions (about which random variables to pick) before observing the outcome of any of them. 

{The main reason behind considering the notion of adaptivity is that 
they could be very hard, sometimes even impossible, to implement in practice due to various reasons. We will discuss this in more detail later, but to name a couple, it could be very costly to wait long enough to observe the outcome of the previous actions, or sometimes it is not possible to measure such outcomes accurately. On the other hand, non-adaptive policies are highly restricted and hence may perform poorly in terms of value compared to adaptive ones. To study this tradeoff between simplicity and value, we study the notion of \textbf{adaptivity gap}, which is defined as the ratio of the expected value of the optimal adaptive policy versus the expected value of the optimal non-adaptive one.}



\subsection*{Contributions}	\label{sec:results}
We show that the adaptivity gap of the stochastic monotone submodular maximization (SMSM) problem is equal to ${e \over e-1} \approx 1.59$. In other words, we prove that there exists a non-adaptive policy that achieves at least an ${e-1 \over e}$ fraction of the value of the optimal adaptive policy. 
We also provide an example to show that our analysis of the adaptivity gap is tight. For that, we use a simple instance of the stochastic max $k$-coverage problem. \upt{Furthermore, we show that the adaptivity gap can be unbounded if the objective function is submodular but non-monotone, see Section~\ref{apx:nonmonotone}.}

We will then present a generalization of the continuous greedy method of~\cite{CalinescuCPV08} and \emph{construct} a non-adaptive policy that approximates the optimal adaptive policy within a factor of $\tfrac{e-1}{e} - \epsilon$ for any arbitrary $\epsilon$. Our algorithm runs in polynomial time in terms of the size of the problem and $1/\epsilon$. This policy does not necessarily coincide with the optimal non-adaptive policy; however, due to the tightness of our adaptivity gap result, this is essentially the best approximation ratio that one could hope for with respect to the optimal adaptive policy.

In Section~\ref{sec:matroidpolicies}, we focus on  myopic policies. We study the natural extension of the myopic policy studied in \cite{FisherNW78-2} to a stochastic environment.
This policy iteratively chooses an element with the maximum \emph{expected} marginal value,
\emph{conditioned} on the outcome of the previous elements.
We show that the approximation ratio of this policy with respect to the optimal adaptive policy is ${1\over 2}$ for {\em general matroids}. In addition, we show that the approximation ratio of this adaptive myopic policy is ${1\over\kappa+1}$ if  the feasible domain is given by the intersection of $\kappa$ matroids.
We also prove that over a uniform matroid (i.e., subject to a cardinality constraint), the approximation ratio of this policy is $\efac$.
We will discuss the results of~\cite{NemhauserW78} and~\cite{Feige98} to show that the approximation ratio of $\efac$ is optimal if only polynomial algorithms are allowed. 

\upt{This paper settles the status of SMSM problem as presented in Table~\ref{table:results}.
The value of the optimal non-adaptive and the optimal adaptive policies are denoted by $\OPTNA$ and $\OPTA$, respectively. The $\oplus$ sign in the table means that the bound is exact, i.e., the lower bound and the upper bound coincide. Also, $\dagger$ means that the algorithm is optimum, i.e., it is computationally intractable to improve the approximation bound. Finally, $\ddagger$ means that our analysis of the algorithm is tight.}

\begin{table}
\begin{center}
\label{table:results}
\caption{Bounds for the Algorithms and the Adaptivity Gap}
\begin{tabular}{ |c | c | c | c |}
\hline
   & & & \\
\textsc{ \textbf{Constraint}} & \textsc{Non-adaptive} & \textsc{Myopic Adaptive} & \textsc{Adaptivity Gap }\\
  & & & \\ \hline \hline
  & & & \\
  \textsc{Uniform} & Myopic Policy & $\tfrac{e-1}{e} \OPTA$  $^\dagger$& \\
 \textsc{Matroid}& $\tfrac{e-1}{2e}\OPTA$ & 
[Theorem~\ref{thm:uniform}] &$\tfrac{e}{e-1}$  $^\oplus$ \\
\cline{1-1}\cline{3-3}
  &[Proposition~\ref{prop:myopic-nonadaptive}] & & [Theorem~\ref{thm:gap}] \\
  \textsc{General} &  & $\tfrac{1}{2} \OPTA$ $^\ddagger$&\\
   \textsc{Matroid}& Stochastic Continuous Greedy &($\tfrac{1}{\kappa+1} \OPTA$  for the  &\\
  &$\tfrac{e-1}{e} \OPTA$  $^\dagger$~~  &  intersection of $\kappa$ matroids)$^\ddagger$ &\\
  & [Theorem~\ref{thm:poly}] &[Theorem~\ref{halfapprox}] & \\ \hline
\end{tabular}
\end{center}
\vspace{1.5em}
\end{table}

\subsection*{Related Work} \label{sec:related_work}
\cite{GoemansV06} proposed adaptive and non-adaptive policies for the stochastic covering problem where the goal is to cover all elements of a target set using stochastic subsets with the minimum (expected) number of subsets. 
\cite{GuilloryB10} studied a similar set cover problem but they took a worst-case analysis approach. They showed that in the worst-case,  the adaptivity gap can be very large and they provide optimal (up to a constant factor) algorithms for their problem.

\upt{Since the publication of our preliminary results in~\cite{AsadpourNS08}, 
adaptive policies for other interesting notions of stochastic submodular maximization have been studied.
~\cite{GolovinK11} introduced 
the notion of {\em adaptive submodularity}.
Adaptive submodularity is a generalization of our setting in which the random variables of interest (namely, $X_1, X_2, \cdots, X_n$) are not necessarily independent. Instead, it suffices that the function $f$ satisfies diminishing conditional expected marginal returns over partial realizations.
They demonstrated the effectiveness and applicability of this notion over a wide range of important optimization problems by providing approximately optimal adaptive algorithms; see Section~\ref{apx:unbouned}.
Also, \cite{GolovinK11_2} 
provided examples to show that the adaptivity gap for certain covering problems (where the objective is to \emph{minimize} the cost) may not be constant.}

Very recently,~\cite{AdamczykSW14} studied the problem of stochastic submodular probing with applications in online dating, kidney exchange, and Bayesian mechanism design. This problem investigates a generalized notion of adaptive policies. In the model, a ground set of elements is given where every element is \emph{active} independently at random with some pre-known probability. A policy can probe elements sequentially, but if an element turns out to be active, it has to be included in the solution. The goal is, naturally, to find a feasible subset of active elements that maximizes the value of a submodular set function. The feasibility of the eventual subset of elements is determined by $k^{{in}}$ matroid constraints. Also, the subset of elements being probed should be feasible and obey $k^{{out}}$ matroid constraints. They provide $(1-1/e)/(k^{{in}}+k^{{out}}+1)$-approximation policies for this problem for $k^{{in}} \geq 0$ and $k^{{out}} \geq 1$. We note that as a complementary result to that paper, when $k^{in} = 1$ and $k^{out} = 0$, the non-adaptive myopic policy provided by our Proposition~\ref{prop:myopic-nonadaptive} provides a $(1-1/e)/(k^{{in}}+k^{{out}}+1) = \frac{e-1}{2e}$-approximation for the stochastic submodular probing problem. Also, the non-adaptive policy resulting from Theorem~\ref{thm:poly} provides a $(1 - 1/e - \epsilon)$-approximation in polynomial time in terms of the size of the problem and $1/\epsilon$.

Another closely related work to ours is by \cite{ChanF08} who studied a rather general stochastic depletion problem. Their optimization problem is defined over a time horizon where at each step a policy decides on an action. Each action stochastically generates a reward and depletes some of the resources. The goal is to maximize the expected reward subject to the resource constraints. They show that under certain ``submodularity" constraints, the myopic policy obtains at least half the value of the optimal (offline) solution {--- a benchmark stronger than the optimal online policy.
An important difference between the framework in \cite{ChanF08} and ours is the ordering imposed in their model on the sequence of actions a policy can take, whereas in our model, in order to prove our results we need to exploit the matroid properties. Although our proposed adaptive polices are not applicable to their setting, our non-adaptive policy can be applied to their setting using a partition matroid to obtain a $\left(1-{1\over e}\right)$-approximation with respect to the optimal adaptive (online) algorithms.}




\subsubsection*{Organization}
The rest of the paper is organized in the following way. In Section~\ref{sec:def} we formally introduce our problem and define adaptive and non-adaptive policies and the concept of adaptivity gap. In Section~\ref{sec:gap} we study the adaptivity gap and show that it is equal to $e/(e-1)$. In Section~\ref{sec:polytight} we provide a polynomial-time algorithm to find a non-adaptive policy that matches the adaptivity gap ratio. The simple myopic policies are investigated in Section~\ref{sec:matroidpolicies}, and worst-case performance guarantees have been provided for them. We discuss some natural variations of our model in Section~\ref{sec:discussion}. Finally, we conclude in Section~\ref{sec:conclusion} by mentioning some interesting directions for future research.

%% file: definition.tex
\section{Model} \label{sec:def}
In this section, we define our framework for stochastic submodular optimization.
Consider a ground set $\A = \{X_1, X_2, \cdots, X_n\}$ of $n$ \emph{independent} random variables.
Let $\Rnonneg$ denote the set of non-negative real numbers (including $0$).  
Each random variable $X_i$ is defined over a domain $\Omega_i \subseteq  \Rnonneg$, and its probability distribution function over $\Omega_i$ has density $g_i$.\footnote{Throughout this paper, we work with continuous probability distributions. All our theorems and their corresponding proofs directly apply to the probability distributions over discrete domains by using Dirac delta functions.}
Let $f: \R_+^n \rightarrow \Rnonneg$ be a {\em monotone submodular value function}.  In other words,
\begin{equation*}
\forall x, y \in \R_+^n: f(x \vee y) + f(x \wedge y) \leq f(x) + f(y),
\end{equation*}
where $x \vee y$ denotes the componentwise maximum and $x \wedge y$ the componentwise minimum of $x$ and $y$.
Note that submodular set functions are special cases of the above definition. This generalization to the continuous domain allows us to capture wider aspects of the problem. For instance, we can model \emph{partial} contributions from the elements. This could be useful in many of the applications that we have mentioned before. For example, in viral marketing, targeted people could have different levels of enthusiasm about the new product or technology. Or, in the sensor placement problem, some sensors in the network might not become fully functional and cover only a part of the area they were supposed to cover.

A {policy} --- formally defined below --- consists of a sequence of actions $(\pi_1, \pi_2, \cdots$). Each {\em action} corresponds to the selection of an element from $\A$ (or perhaps none, denoted by $\emptyset$, corresponding to the empty set). After an element is selected, its actual value is realized. The state of the problem $\Theta$ is represented by an $n$-tuple $(\theta_1, \theta_2, \cdots, \theta_n)$.
Here, $\theta_i = x_i$ if $X_i$ has been selected by the policy in any of the previous steps and its actual value is realized to be $x_i$; otherwise $\theta_i = \nill$. The symbol $\nill$, different from value zero, stands for null, representing that the corresponding element has not been selected yet. 

Let $\Omega := \prod_{i=1}^n(\Omega_i \cup \{\nill\})$ be the set of all possible states. Consider a policy $\pi$ where the state of the problem before taking the $j$-th action is given by $\Theta^{\pi}_{j-1} \in \Omega$. The $j$-th action, denoted by $\pi_j$, 
is a mapping from the current state to an element from set $\A \cup \{\emptyset\}$.
Before taking any action the state of the problem is $\Theta^{\pi}_0 = (\nill, \nill, \cdots, \nill)$.
For instance, if we choose to pick $X_3$ in our first action (i.e., $\pi_1 = X_3$) and its value is realized to be $x_3$, then the state will be $\Theta^{\pi}_1 = (\nill, \nill, x_3, \nill, \cdots, \nill)$. If at some step $j$ the action $\emptyset$ is selected, then the state of the problem will remain unchanged, i.e., $\Theta^{\pi}_{j} = \Theta^{\pi}_{j - 1}$. We can now formally define the notion of policy.

\begin{definition}[\textsc{Policy}] \label{def:policy}
A policy $\pi:\Omega \rightarrow [0,1]^n$ is a mapping from the state space $\Omega$ to a distribution over the elements in $\A \cup \{\emptyset\}$. Let $\Theta^{\pi}_{j-1} = (\theta_1, \theta_2, \cdots, \theta_n)$ represent the state of the problem before taking the $j$-th action. Then, $\pi(\Theta^{\pi}_{j-1})$ is a distribution over $\{X_k|\theta_k = \nill\} \cup \{\emptyset\}$.  If the policy is deterministic, $\pi(\Theta^{\pi}_{j-1})$ will correspond to exactly one element in $\{X_k|\theta_k = \nill\} \cup \{\emptyset\}$. The action of the policy, denoted by $\pi_j$, corresponds to choosing an element of $\{X_k|\theta_k = \nill\} \cup \{\emptyset\}$ according to distribution $\pi(\Theta^{\pi}_{j-1})$. Subsequently, if the policy chooses element $X_i\neq\emptyset$, the state of the problem after taking the $j$-th action will be $\Theta^{\pi}_{j} = (\theta_1, \theta_2, \cdots, \theta_{i - 1}, x_i, \theta_{i+1}, \cdots, \theta_n)$, where $x_i$ is the realized value of $X_i$  (drawn from distribution $g_i$). Finally, if at any point $\pi_j(\Theta^{\pi}_{j-1}) = \emptyset$, then the selection of the elements stops. {We denote by $\Theta^\pi$ the random $n$-tuple corresponding to the last state reached by the policy.}
%
\end{definition}



In the following example, we discuss the evolution of a simple policy defined on a ground set consisting of three elements.

\paragraph{Example:}
Suppose that the ground set of elements is $\{X_1, X_2, X_3\}$. Also, suppose that each element, once picked, will have a realized value of either 10 or 100 with probabilities $0.4$ and $0.6$, respectively. Consider a simple policy $\pi$ that in the first step picks element $X_3$. In the next step, $\pi$ will pick $X_1$ if the realized value of $X_3$ is realized to be 10. Otherwise, it picks either $X_1$ or $X_2$, each of which with probability $0.5$. After that, $\pi$ does not pick anything. Note that as we have explained before, the initial state of the problem is by definition $\Theta^{\pi}_{0} =(\nill, \nill, \nill)$. The first action of this policy is $\pi_1 = X_3$. After that, the state of the problem, i.e., $\Theta^{\pi}_{1}$, will be either $(\nill, \nill, 10)$ or $(\nill, \nill, 100)$ with probabilities $0.4$ and $0.6$, respectively. 
Now, if $\Theta^{\pi}_{1} = (\nill, \nill, 10)$, then the second action $\pi_2(\nill, \nill, 10)$ will be to choose $X_1$. Consequently, the state of the problem afterwards, i.e., $\Theta^{\pi}_{2}$, will be either $(10, \nill, 10)$ or $(100, \nill, 10)$ with probabilities $0.4$ and $0.6$, respectively. On the other hand, if $\Theta^{\pi}_{1} = (\nill, \nill, 100)$, then the second action will be a random one that picks either $X_1$ or $X_2$, each of which with probability $0.5$. In the first case, the state of the problem, i.e., $\Theta^{\pi}_{2}$, will be either $(10, \nill, 10)$ or $(100, \nill, 10)$ with probabilities $0.4$ and $0.6$, respectively. Similarly, in the second case, $\Theta^{\pi}_{2}$ will be either $(\nill, 10, 10)$ or $(\nill, 100, 10)$ with probabilities $0.4$ and $0.6$, respectively. 

From the discussion above, we conclude that for this specific policy $\pi$ we have
\begin{equation}
\label{Eq:example}
\Theta^{\pi}_{2} = \left\{\begin{array}{ll} (10, \nill, 10) & \mbox{w.p. $0.4 \times 0.4 = 0.16$}, \\ 
(100, \nill, 10) & \mbox{w.p. $0.4 \times 0.6 = 0.24$}, \\
(10, \nill, 100) & \mbox{w.p. $0.6 \times 0.5 \times 0.4 = 0.12$},\\
(100, \nill, 100) & \mbox{w.p. $0.6 \times 0.5 \times 0.6 = 0.18$}, \\
(\nill, 10, 100) &\mbox{w.p. $0.6 \times 0.5 \times 0.4 = 0.12$},\textrm{~and} \\
(\nill, 100, 100) & \mbox{w.p.  $0.6 \times 0.5 \times 0.6 = 0.18$}. \end{array} \right.
\end{equation}

After that, the policy stops. {Hence, $\Theta^{\pi}$ will be equal to $\Theta^{\pi}_{2}$.} 
$\hfill\square$
\medskip



Note that a policy sequentially picks new elements from the ground set $\A$ (perhaps, based on the outcomes of the previous selections). 
We emphasize that we only consider policies \emph{without} replacements (i.e., only elements whose corresponding entries in the current state are null can be selected). As a consequence, once the value of an element is realized, it will be fixed and cannot be changed by picking that element again. 
As we discuss in Section~\ref{sec:mult}, this assumption is only made to simplify the presentation, and all of our results hold without it.

The value of any state $\Theta = (\theta_1, \theta_2, \cdots, \theta_n)$, denoted by $\ff(\Theta)$, is equal to the value of function $f$ over those elements of $\A$ that already have been chosen. Formally,
\begin{equation}
\label{eq:value}
\ff(\Theta) := f(\zeta_1, \zeta_2, \cdots, \zeta_n), \textrm{~~~~~~where }
\zeta_i = \left\{\begin{array}{ll} \theta_i & ~~~\theta_i \neq \nill; \\ 0 & ~~~\theta_i = \nill. \end{array}\right.
\end{equation}

For instance, for a state $\Theta = (10, \nill, 100)$, the value of $f(\Theta)$ is equal to $f(10, 0, 100)$.
Finally, \emph{the expected value} of a policy $\pi$ after the $j$-th action is equal to $\E[\ff(\Theta^{\pi}_j)]$.
Here, the expectation is taken over all possible realizations of the state $\Theta^{\pi}_j$ as an outcome of the first $j$ actions taken by $\pi$. {Throughout this paper the parameter for the expectation will be clear from the context. We will make this parameter explicit wherever some confusion may arise. For instance, if needed, we can represent the aforementioned expectation by $\E_{\Theta}[\ff(\Theta^{\pi}_j)]$.} 

\begin{definition}[\textsc{Value of a Policy}]
The  value of a policy $\pi$ is defined as the expected value of its final state, i.e., $\E[\ff(\Theta^{\pi})]$. \end{definition}
%

In the above definition, the expectation is taken over all realization of $\Theta^{\pi}$. We refer the reader to the end of Definition~\ref{def:policy} for $\Theta^{\pi}$. Thus, for the example we discussed earlier, by following Eq.~(\ref{Eq:example}) we have
\begin{eqnarray*}
\E[\ff(\Theta^{\pi})] = \E[\ff(\Theta^{\pi}_2)] &=& 0.16f(10, 0, 10) + 0.24f(100, 0, 10) + 0.12f(10, 0, 100)\nonumber \\
&& +~ 0.18f(100, 0, 100) + 0.12f(0, 10, 100) + 0.18f(0. 100, 0). \nonumber
\end{eqnarray*}

Throughout this paper, we assume that all such expected values are finite. Due to the monotonicity of $f$, it is enough to assume that $\E[\ff(\Theta^{\pi})]$ is finite for a policy $\pi$ that picks all the elements in $\A$. 

Finally, we take into account the constraints that might appear in our optimization problems.
\begin{definition}[\textsc{A Feasible Policy}]
Consider a matroid $\M(\A, \I)$ defined on a ground set $\A = \{X_1, X_2, ... , X_n\}$ of $n$ random variables.
A policy $\pi$ is feasible with respect to $\M$ if and only if with probability $1$ the set of elements chosen by $\pi$ belongs to $\I$. {We denote by $\Pi(\A,\I)$ the set of all such feasible policies.} 
\end{definition}

Now, we are ready to formally define the problem of maximizing a stochastic monotone submodular function in the presence of a matroid feasibility constraint.

\begin{definition}[SMSM]
{\textbf{\em Maximizing a Stochastic Monotone Submodular Function with Matroid Constraint:}}
{An instance of the SMSM problem is defined with tuple $\big(f,\A,\I\big)$, where  $\A = \{X_1, X_2, \cdots, X_n\}$ is a set of $n$ non-negative independent random variables, $f: \R_+^n \rightarrow \Rnonneg$ is a monotone submodular value function, and matroid $M(\A,\I)$ represents the constraints. 
The goal is to find a feasible policy (with respect to $\M$) that obtains the maximum (expected) value, {i.e., 
\textbf{$\max_{\pi \in \Pi(\A, \I)} \E[f(\Theta^{\pi})].$}}}
\end{definition}
\medskip

In addition to ``adaptive" policies, we also consider \emph{non-adaptive} policies that do not rely on the outcomes of the previous actions in order to make a decision.
When a policy progresses, more information about the actual realization of the state will be revealed.
Non-adaptive policies do not use such information and choose all the elements in advance.

\begin{definition}[Non-Adaptive Policies] 
A {policy} $\pi$ is called {\em non-adaptive} if, for each state $\Theta$, $\pi(\Theta)$ depends only on which elements were previously chosen but {\em not} their realizations. A non-adaptive policy is defined by the set (and not the order) of the chosen elements.
\end{definition}
%

Since non-adaptive policies choose their set of actions in advance (not sequentially), they are usually simpler to implement. Furthermore, sometimes it is not possible to implement an adaptive policy or it might be difficult to learn the realizations of the random variables. For instance, in the capital budgeting problem, it might be costly for the investor to wait for outcomes of the projects to be realized before making another decision. In the case of viral marketing, it is difficult to get feedback from each individual and find out about the outcome of the previous actions.

However, one drawback of non-adaptive policies is that they may not perform as well as the adaptive policies. The optimal policy for a given instance of the SMSM problem might be a complicated adaptive policy, far from being non-adaptive. Hence, there exists a trade-off between the value obtained by the policy on one side and its practicality and convenience (in terms of both computation and  implementation) on the other side. However, contemplating between these two options is reasonable for strategists only if the performance of the non-adaptive policies is not too far from that of the adaptive ones.
This observation motivates the comparison between the value of adaptive and non-adaptive policies in SMSM problems. This can be measured via studying the \emph{adaptivity gap} as defined by~\cite{DeanGV05, DeanGV04}.
\begin{definition}[Adaptivity Gap] \label{def:gap}
The adaptivity gap is defined as the upper bound on the ratio of the value of the optimal adaptive policy to the value of the optimal non-adaptive policy. More precisely, 
$$\textrm{\textsc{Adaptivity Gap}} = \sup_{(f,\A,\I)}  \frac{\sup_{\pi \in \Pi(\A,\I)} \E_{\Theta}[\ff(\Theta^{\pi})]}{\sup_{S\in\I} \E_{\Theta}[\ff(\Theta_S)]},$$
where the first supremum is taken over all instances $(f,\A,\I)$ of the SMSM problem, and for $S\in I$ the vector $\Theta_S$ corresponds to the outcome of the non-adaptive policy that picks set $S$ of the elements.
\end{definition}


%% file: gap.tex
\section{The Adaptivity Gap of SMSM}  \label{sec:gap}
It is easy to observe from Definition~\ref{def:gap} that the adaptivity gap is at least equal to $1$ simply because non-adaptive policies are a special case of adaptive (general) policies. However, a high adaptivity gap for a problem suggests that good adaptive policies are far superior than the non-adaptive ones. Similarly, a low (close to 1) adaptivity gap is a certificate that one can get the benefits of non-adaptive policies (such as the relative simplicity in their optimization and also implementation) without losing much in the obtained value compared to that of the adaptive policies. 
In this section, we show that for the SMSM problem, the adaptivity gap is bounded.
The main result of this section is as follows:

%
\begin{theorem}[Adaptivity Gap] \label{thm:gap}
The adaptivity gap of SMSM is equal to $\frac{e}{e-1}$.
\end{theorem}

In the rest of this section, we prove the above theorem by analyzing  the optimal non-adaptive and adaptive policies and compare their performances. We start with an example that show that the above ratio is tight.

\subsection{A Tight Example: Stochastic Maximum $k$-Coverage} \label{sec:tightexample}
%
Given a collection $\F$ of the subsets of $\{1,2,\cdots,m\}$,
the goal of the {\em maximum $k$-coverage} problem is to find $k$ subsets from $\F$ such that their union has the maximum cardinality (see~\cite{Feige98}).
In the stochastic version, each random variable $X_i$ with some probability covers a subset of $\{1,2,\cdots,m\}$ and with the remaining probability covers no element 
(i.e., the realization of $X_i$ would be the empty set). 
The subset that an element of $\F$ would cover is revealed only after choosing the element
according to a given probability distribution.
It is easy to see that  this problem is a special case of the SMSM
with respect to a uniform matroid $\M = (\F, \{S \subseteq \F: |S| \leq k\})$ constraint. 

To give a lower bound on the adaptivity gap,
consider the following instance: a ground set $\{1, 2, \cdots, m\}$ and a
collection $\F = \{X^{(i)}_j|1 \leq i \leq m, 1 \leq j \leq m^2\}$ of its subsets are given. Note that here $n$, the total number of elements in the ground set, is equal to $m^3$. For every $i, j$, define $X^{(i)}_j$ to be the one-element subset $\{i\}$ with probability ${1\over m}$ and the empty set with probability $1 - {1\over m}$.
The goal is to cover the maximum number of the elements of the ground set by selecting at most $k := m^2$ subsets from $\F$. We denote this instance by SMSM1($m$).
\begin{lemma}
\label{lem:example_nonadaptive}
The optimal non-adaptive policy for SMSM1($m$) is to pick $m$ subsets from each of the collections $\F^{(i)} = \{X^{(i)}_j|1 \leq j \leq m^2\}$ for every $i$. The expected value of this policy is $\left(1-{1\over e} - o(1)\right)m$.
\end{lemma}

{Therefore, for large enough values of $m$, the expected value of the non-adaptive policy above is arbitrarily close to $(\efac)m$.}  The proof is given in the online appendix and is based on a convexity argument. 

We now consider the following myopic adaptive policy $\pi$:
start with $i = 1$ and pick the elements of $\A^{(i)}$ one by one until one of them is realized as $\{i\}$ or all of elements in $\A^{(i)}$ are chosen. Then increase $i$ by one. Continue the iteration until either all the elements are covered or $m^2$ subsets from $\A$ are selected. 
The following lemma gives a lower bound on the number of elements in the ground set covered by the adaptive policy. The proof is given in the online appendix.

\begin{lemma}
\label{lem:example_adaptive}
The expected number of elements in $\{1,2,\cdots,m\}$ covered by $\pi$ described above is $(1 - o(1))m$.
\end{lemma}

By combining the results of Lemmas~\ref{lem:example_nonadaptive} and~\ref{lem:example_adaptive},
we have that the adaptivity gap of SMSM1($m$) is at least $\tfrac{e}{e - 1}$.


\subsection{Bounds on the Optimal Non-Adaptive Policy} \label{sec:bounds_non_adapt}
In this section, we study the optimal non-adaptive policy in more detail.
We start with some definitions.
Let $S \subseteq \A$ be a subset of variables.
Also, let vector $\Theta_S = (\theta_1, \cdots, \theta_n)$ denote the {realization} of set $S$,
where $\theta_i = x_i$ for $X_i\in S$ and $\theta_i = \nill$ for $X_i\notin S$.
The value obtained by choosing set $S$ after the realization is equal to $\E[\ff(\Theta_S)]$.
Let $g_i$ and $G_i$ be the probability density function (p.d.f.) and cumulative distribution function (c.d.f.) of random variable $X_i$, respectively.
We introduce $g_S:\R_+^n \rightarrow \R$
to represent the probability density function of observing $\Theta = (\theta_1, \cdots, \theta_n)$ while selecting $S$:
\begin{eqnarray*}
g_S(\Theta) := \Pr[\Theta_S = \Theta] = \prod_{i:X_i\in S} g_i(\theta_i),
\end{eqnarray*}
where $g_S(\theta)$ is defined to be zero if there exists $j\notin S$ such that $\theta_j \neq \nill$.
Note that the definition above relies on the assumption that the realizations of elements chosen by the policy are independent from each other. As we show in Section~\ref{apx:unbouned},
the adaptivity gap can be arbitrarily large if this assumption does not hold.

Now, for the sake of simplicity in notation, we define function $F: 2^{\A} \rightarrow \R^+$ as the expected value
obtained by choosing set $S$, i.e.,
\begin{eqnarray}
\label{eq:F} F(S) = \E[\ff(\Theta_S)] = \int_{\Theta \in \Omega}\ff(\Theta) g_S(\Theta) d\Theta,
\end{eqnarray}
where, as defined before, $\Omega = \prod_{i = 1}^n (\Omega_i \cup \{\nill\})$ is the set of all possible states. Therefore, the optimal non-adaptive policy for SMSM is equivalent to choosing
a set $S\in \I$ that maximizes $F(S)$ with respect to the desired matroid constraint. Hence, the following proposition holds:

\begin{proposition}[Optimal Non-Adaptive Policy] \label{prop:non-adapt}
For any instance of SMSM, each set $S^{\star} \in \mbox{\em argmax}_{S\in\I} F(S)$ corresponds to an optimal non-adaptive policy.
\end{proposition}

In the online appendix we prove the lemma below that shows  $F$ is monotone and submodular.

\begin{lemma}[Properties of Function $F$]  \label{lem:submodularity} 
If $f: \R_+^n \rightarrow \Rnonneg$ is monotone and submodular, then function $F: 2^{\A} \rightarrow \R^+$, defined in {\em Eq.~$($\ref{eq:F}$)$}, is a monotone and submodular {\em set function}.
\end{lemma}

We now define the notion of {\em fractional non-adaptive policies.}  A fractional non-adaptive policy ${\cal S}_y$ is defined with a vector $y \in [0,1]^n$. Policy ${\cal S}_y$ chooses elements in the (random) set $Y$; a set that includes each $X_i\in\A$ with probability $y_i$, independently for every $i$. 

We call fractional non-adaptive policy ${\cal S}_y$ {\em feasible in expectation} if vector $y$ lies in the base polytope of $\M$ denoted by $\B(\M)$.  The base polytope of $\M$ is the convex hull of the characteristic vectors of all bases of $\M$, i.e.,
\begin{equation} \label{eq:bm}
\B(\M) = \textrm{conv}\{1_S | S\in\I, \textrm{ and } S \textrm{ is a basis} \}.
\end{equation}

For instance, suppose $\M$ is a uniform matroid of rank $d$, and for $y = (y_1,y_2,\cdots,y_n)$ we have $\sum_{i=1}^n y_i \le {d}$. Observe that in this case,  fractional non-adaptive policy ${\cal S}_y$ is feasible in expectation, and the expected number of elements chosen by this policy is at most $d$.

With slight abuse of notation, we denote the expected value obtained by fractional non-adaptive policy ${\cal S}_y$ by $F(y)$. Namely, 
\begin{eqnarray} \label{expancion}
 F(y) :=  \sum_{Y \subseteq \{0, 1\}^n} \left[\left(\prod_{i \in Y} y_i \prod_{i \notin Y} (1 - y_i) \right) F(Y)\right],
\end{eqnarray}
where as defined in Eq.~(\ref{eq:F}), we have $F(Y) =  \E_{\Theta}[\ff(\Theta_Y)]$. 

Fractional non-adaptive polices extend the space of (integral) non-adaptive polices. Therefore, it is easy to see that $\max_{S\in\I} F(S) \le \max_{y\in\B(\M)} F(y)$. We show that this inequality is in fact tight, and any vector $y\in \B(\M)$ can be rounded to an integral corner point $Y \in \B(\M)$ corresponding to a subset $S\in\I$ such that $F(S) \ge F(y)$. \upt{Note that Lemma~\ref{lem:submodularity} shows that function $F(Y)$ used in Eq.~(\ref{expancion}) is a \emph{submodular set function}. Hence, $F(y)$ is the same as the ``multilinear extension'' that has been defined by~\cite{CalinescuCPV08} to study the problem of maximizing (general deterministic) submodular functions. Therefore, by applying the pipage rounding procedure of~\cite{CalinescuCPV08} to function $F(y)$ we immediately obtain the following lemma.}

\begin{lemma}[Fractional vs. Integral Non-Adaptive Policies] \label{lem:fractional_non_adapt}
For any $y\in\B(\M)$, there exists a set $S\in\I$ such that $F(S) \ge F(y)$. Moreover, such a set can be found in polynomial time.
\end{lemma}

The above lemma implies that $\max_{y \in \B(\M)} F(y)$ is a (tight) \emph{lower bound} on the performance of the optimal non-adaptive policy. 

\subsection{Bounds on the Optimal Adaptive Policy}

We start this section by making a few observations about adaptive policies. Consider an arbitrary adaptive policy \textsc{Adapt}. Any such policy decides to choose a sequence of elements where the decision on which element to choose at any step might depend on the realized values of the previously chosen elements. {Therefore, any specific state of the world will result in a distribution over the sequence of elements that will be chosen by \textsc{Adapt}.}\footnote{We refer the reader to Definition~\ref{def:policy} for policy and emphasize that we allow our policies to be \emph{randomized}. This is the reason that we mention a distribution over different sequences of elements (and not just a specific sequence) here.}
 Any adaptive policy can be described by a (possibly randomized) decision tree in which at each step an element is being added to the current selection.
Each path from the root to a leaf of this tree corresponds to a subset of elements and occurs with some certain probability.
{These probabilities, covering all the possible scenarios for \textsc{Adapt}, sum to one.}
Let $y_i$ be the probability that element $X_i\in\A$ is eventually chosen by \textsc{Adapt}.
Also, let $\beta_\Theta$ be the probability that the final state $\Theta$ is reached by \textsc{Adapt}. Then we have the following properties:
\medskip

\begin{enumerate}
\item $\displaystyle \int_{\Theta\in\Omega} \beta_{\Theta}d\Theta = 1$.
\medskip
\item $\forall \Theta \in \Omega: \beta_{\Theta} \geq 0$.
\medskip
\item $\displaystyle  \forall i, x_i\in \Omega_i: \int_{\Theta: \theta_i  = x_i} \beta_{\Theta} d \Theta^{-i} = y_i g_i(x_i)$,
\medskip
\end{enumerate}

where $d\Theta^{-i}$ represents $\prod_{j \neq i} d\theta_j$. The first two properties hold because $\beta$ defines a probability measure on the space of all feasible outcomes.
The third property implies that the probability that we observe an outcome $x_i$ as a realized value of $X_i$ among all possible states $\Theta$ reached by the policy
is equal to the probability that $X_i$ is chosen (i.e., $y_i$) multiplied by the probability that the realization is equal to $x_i$.
We emphasize that we use the independence among the random variables to ensure that this property holds. We also emphasize that, as we will see in Section~\ref{apx:unbouned}, the adaptivity gap could be unbounded if there are dependencies among the variables.

Since every policy satisfies the above properties,
we can establish an upper bound on the value of any adaptive policy.
Hence, we define the function $f^+:[0,1]^n \rightarrow \mathbb{R}$ as follows:
\begin{equation}
\label{extensionintegral}
\displaystyle  f^+(y) := \sup_{\alpha}\left\{\int_{\Theta} \alpha_{\Theta} \ff(\Theta)d \Theta: \displaystyle \int_{\Theta} \alpha_{\Theta}d \Theta = 1, \alpha_{\Theta} \geq 0, \forall i, x_i\in\Omega_i: \int_{\Theta: \theta_i = x_i} \alpha_{\Theta} d\Theta^{-i} = y_i g_i(x_i)\right\}.
\end{equation}
The supremum is taken over all probability measures $\alpha$ that satisfy the three properties above.

Another observation is that for an optimal adaptive policy, vector $y$ lies in the base polytope of $\M$; see Eq.~(\ref{eq:bm}). Using these observations, we obtain the following lemma which is proved in the appendix:

\begin{lemma}[An Upper Bound on Adaptive Policies] \label{adaptiveupperbound}
The expected value of the optimal adaptive policy is at most $\sup_{y \in \B(\M)}\{f^+(y)\}$. 
\end{lemma}

%

We are now ready to prove Theorem~\ref{thm:gap}.

\subsection{Proof of Theorem~\ref{thm:gap}}
Lemma~\ref{adaptiveupperbound} shows that $\sup_{y \in \B(\M)} f^+(y)$ is an upper bound on the performance of the optimal adaptive policy.
Now consider any $y \in \B(\M)$. Below, we present Lemma~\ref{lem:upperbound}, which shows $\left(1-{1\over e}\right) f^+(y) \le  F(y)$. On the other hand, Lemma~\ref{lem:fractional_non_adapt} implies that there exists a set $S \in \I(\M)$ such that $F(y)\le F(S)$. Recall that $F(S)$ is in fact the expected value gained by a non-adaptive policy that selects set $S$. Hence, for every $y \in \B(\M)$, there exists a non-adaptive policy that achieves a value of at least $(1 - \frac{1}{e})f^+(y)$. Therefore, the optimal non-adaptive policy corresponding to selecting set $S^* \in \argmax_S \{F(S)\}$ has a value of at least  $(1 - \frac{1}{e}) \sup_{y \in \B(\M)}f^+(y)$, or equivalently, at least a $(1-1/e)$ fraction of the value of the optimal adaptive policy.
We remind the reader that the example in Section~\ref{sec:tightexample} shows that the factor is tight.

\begin{lemma}[Upper Bound on $f^+$]\label{lem:upperbound}
For any monotone submodular function $f$ and any vector $y \in \B(\M)$, we have $f^+(y) \le (\tfrac{e}{e-1}) F(y).$
\end{lemma}

The main technical ingredient in the proof of the lemma above is a generalization of Lemmas 3.7 and 3.8 in~\cite{VondrakThesis07} to continuous submodular functions for which we provide a new stochastic dominance result.

\begin{proof}{Proof of Lemma~\ref{lem:upperbound}}
%
We start with introducing the following notation. For a vector $\Theta = (\theta_1, \theta_2, \cdots, \theta_n)$, define $\Theta^{\uparrow j} = (\theta^{\uparrow j}_1, \theta^{\uparrow j}_2, \cdots, \theta^{\uparrow j}_n)$ as a {random} vector whose entries are defined as below.
\begin{equation} \label{eq:magic}
\theta^{\uparrow j}_i := \left\{\begin{array}{ll} \theta_i & ~~~i \neq j; \\ \max\{\theta_j, X\} & ~~~i = j, \textrm{~~where~$X$ is independently drawn from $g_j$}. \end{array}\right.
\end{equation}
One can think of $\Theta^{\uparrow j}$ as the same as $\Theta$, except for the $j$-th entry for which we draw a number $X$ from the distribution $g_j$ and override the entry with $X$ if and only if it increases the value of the entry.

Now, for any $j$ we define an independent Poisson clock $\C_j$ that sends signals with rate $y_j$ throughout the time. We start with $\Theta = (0, 0, \cdots, 0)$ at $t=0$. Once a clock $\C_j$ sends a signal, we replace $\Theta$ by $\Theta^{\uparrow j}$. By abuse of notation, we denote
the value of vector $\Theta$ at time $t$ by $\Theta(t)$.
We first show that $\E[\ff(\Theta(1))] \leq F(y)$. We will do so by proving a stochastic dominance result. 

Consider $F(y)$, defined in Eq.~(\ref{expancion}). Note that each entry $j$ in $\Theta_Y$ is zero (null) if $j \notin Y$. Otherwise, it is drawn from the probability distribution $g_j$, independently from other variables. 
Hence, $F(y)$ can be written as $E[f(\theta_1, \theta_2, \cdots, \theta_n)]$, where the entries are all independent and the c.d.f. of $\theta_j$ is given by a function $\eta_j:\R_+ \mapsto [0,1]$ and $\eta_j(x) = (1- y_j) + y_j G_j(x)$.

On the other hand, due to the independence of the Poisson clocks, the entries of $\Theta(1)$ are independent random variables. In particular, the clock $\C_j$  signals $k$ times between $t=0$ and $t=1$, where $k$ is a Poisson random variable with rate $y_j$. If this clock signals $k$ times, by the construction of $\Theta(t)$ for $0 \leq t \leq 1$, the $j$-th entry of $\Theta(1)$ will be the maximum of $k$ random variables, each drawn independently from the probability distribution $g_j$. In this case, the c.d.f. of the $j$-th entry will be simply given by the function $G_j(x)^k$. Summing over all $k$ and incorporating the Poisson distribution function, we can summarize that the c.d.f. of the $j$-th entry of $\Theta(1)$ is given by a function $\gamma_j:\R_+ \mapsto [0,1]$, where $\gamma_j(x) = \sum_{k=0}^{\infty} \tfrac{y_j^k}{k!}e^{-y_j}G_j(x)^k$. 

By the properties of Poisson distribution, we have 
\begin{eqnarray*}
\gamma_j(x) & = & \sum_{k=0}^{\infty} \frac{y_j^k}{k!} e^{-y_j} G_j(x)^k = \frac{e^{-y_j}}{e^{-y_j G_j(x)}} \sum_{k=0}^{\infty} \biggl(\frac{(y_j G_j(x))^k}{k!} e^{-y_j G_j(x)} \biggl) = \frac{e^{-y_j}}{e^{-y_j G_j(x)}} = e^{-y_j(1-G_j(x))} 
\end{eqnarray*}

However, comparing the two cumulative distribution functions discussed above, we have 
\begin{eqnarray*}
\gamma_j(x) & = & e^{-y_j(1-G_j(x))}  \ge  1 - y_j(1-G_j(x)) = \eta_j(x).
\end{eqnarray*}

This means that for every $j$, the random variable drawn from $\gamma_j$ is stochastically dominated by that from $\eta_j$. Therefore,  
\begin{equation}
\label{firstneq}
\E[\ff(\Theta(1))] \leq F(y).
\end{equation}

Now, we compare the value $\E[\ff(\Theta(1))]$ with $f^+(y)$. Let $t \in [0, 1)$ be fixed. For each $j$, the chance that the clock $C_j$ sends a signal during the interval $[t, t+dt)$ is simply $y_j dt$ for sufficiently small $dt$. Hence, 
\begin{equation}
\label{secondneq}\E[\ff(\Theta(t + dt)) - \ff(\Theta(t))~ |~ \Theta(t) = \Theta] = \sum_{j = 1}^n  y_j dt \left(\E[f(\Theta^{\uparrow j})] - f(\Theta)\right).
\end{equation}
Note that the expectation is taken over the random draw that appears in the definition of $\Theta^{\uparrow j}$, i.e., the random variable $X$ in Definition~(\ref{eq:magic}). 

We define an auxiliary function $f^*:[0,1]^n \rightarrow \mathbb{R}$ as the following:
\begin{equation}\label{eq:defstar}
f^*(y) := \inf_\Theta\left\{\ff(\Theta) + \sum_{j = 1}^{n} y_j \Big(\E\left[f(\Theta^{\uparrow j})\right] - f(\Theta)\Big)\right\},
\end{equation}
where the infimum is taken over all possible states $\Theta$. Therefore, the right hand side of Eq.~(\ref{secondneq}) is at least $(f^*(y) - \E[\ff(\Theta)])dt$.
As a result, the following bound can be derived on the derivative of $\E[\ff(\Theta(t))]$:
\begin{eqnarray} \label{eq:derivative}
\frac{d}{dt}\E[\ff(\Theta(t))] = \frac{1}{dt}\E[\ff(\Theta(t + dt)) - \ff(\Theta(t))~ |~ \Theta(t) = \Theta] &\geq& f^*(y) - \E[\ff(\Theta(t))].
\end{eqnarray}
Solving the differential equation above we obtain $\E[\ff(\Theta(t))] \geq (1 - e^{-t})f^*(y)$. In particular, 
\begin{equation}
\label{thirdneq}
\left(1 - \frac{1}{e}\right)f^*(y) \le \E[\ff(\Theta(1))].
\end{equation}

In Appendix~\ref{app:gap}, Lemma~\ref{lem:fplusstar}, we show that $f^+(y)\le f^*(y)$, $y\in\B(\M)$.  
Combining with Inequalities~(\ref{firstneq}) and~(\ref{thirdneq}), we have $f^+(y) \leq \tfrac{e}{e-1}F(y)$, which concludes the proof of the lemma.
\qed\end{proof}

\input{polynomial}

%% file: polynomial.tex
\section{A Polynomial-Time $(1- 1/e - \epsilon)$-Approximate Non-Adaptive Policy}\label{sec:polytight}
In this section, we present polynomial-time non-adaptive policies for the SMSM problem. We assume that the values of function $F$ are accessible via an oracle.
We discuss the computation of function $F$ in Section~\ref{sec:f_computation}. We show that under standard assumptions (when $g_j$'s are bounded or constant Lipschitz continuous), the value of function $F$ can be computed within any polynomially small error term using sampling.

Let us first consider the following myopic non-adaptive policy formally defined in Figure~\ref{fig:nonadaptivepolicy}. This policy, at each step, adds a feasible element with the highest marginal contribution to the set of selected elements. Recall that using Lemma~\ref{lem:submodularity}, we showed that function $F$ is monotone submodular. 

The classic result of \cite{FisherNW78-2} implies that the greedy algorithm obtains an approximation ratio of ${1\over 2}$ for maximizing monotone submodular functions, i.e., at the end of the algorithm $F(S_i) \ge {1\over 2} \max_{S\in\I} \{F(S)\}$. 
Therefore, using Theorem~\ref{thm:gap}, we have

\begin{proposition}[Myopic Non-Adaptive Policy] \label{prop:myopic-nonadaptive}
The myopic non-adaptive policy for any instance of SMSM obtains an expected value of at least ${1\over 2}\times (1-{1\over e}) \approx 0.316$ times the expected value of the optimal adaptive policy.
\end{proposition}

\def\i{\hspace*{7mm}}
\begin{figure}[!t]
\caption{The Non-Adaptive Myopic Policy}
\begin{center}
\fbox{\parbox{5in}{ \vspace*{2mm} 
\label{fig:nonadaptivepolicy}

\begin{description}
\item\i\i Initialize $i = 0, S_{0} = \emptyset, U_{0} = \emptyset$.
\item\i\i Repeat
\item\i\i\i $i \leftarrow i + 1$. 
\item\i\i\i   Find $X_i\in\argmax_{X \in \A \setminus (U_{i-1} \cup S_{i-1})} \E[F(S_{i-1}\cup \{X_i\})]$.
\item\i\i\i   If $S_{i-1} \cup \{X_i\} \in \I$, then
\begin{description}
\item\i\i\i\i $S_{i} \leftarrow S_{i-1} \cup \{X_i\}$; $U_{i} \leftarrow U_{i-1}$.
\end{description}
\item\i\i\i   else
\begin{description}
\item\i\i\i\i $U_{i} \leftarrow U_{i-1} \cup \{X_i\}$; $S_{i} \leftarrow S_{i-1}$.
\end{description}
\item\i\i Until $(U_{i} \cup S_{i} = \A)$. 
\item\i\i The non-adaptive policy selects set $S_i$.
\end{description} 
}
\medskip
}
\end{center}
\medskip
\end{figure}


The above result provides a performance guarantee for a simple policy; however, it does not match the adaptivity gap. Theorem~\ref{thm:gap} shows that there exists a non-adaptive policy corresponding to the selection of set $S^* \in\argmax_{S \in \I(\M)} F(S)$ that achieves at least a $\left(1 - 1/e\right)$ fraction of the value of the optimal adaptive policy. However, this does not immediately provide an efficient algorithm to find such a non-adaptive policy. The reason is that finding $\argmax_{S \in \I(\M)} F(S)$, i.e., maximizing a monotone submodular function with respect to a matroid constraint, is not computationally tractable. As mentioned before, the maximum $k$-cover, described in Section~\ref{sec:tightexample}, is a special case of this problem. \cite{Feige98} shows that it is not possible to find an approximation ratio better than $\efac$ for the maximum $k$-cover problem unless $\textrm{NP} \subset \textrm{TIME}(n^{O(\log \log n)})$. Hence, $\efac$ is the best possible approximation ratio achievable for any policy that can be implemented within polynomial time.



In the proof of Theorem~\ref{thm:gap}, we used a dynamic process $\Theta(t): t \in [0, 1]$ to find a vector $y^*$ such that $F(y^*) \geq (1-1/e) \sup_{y \in \B(\M)} f^+(y)$. 
This process takes an arbitrary $y \in \B(\M)$ and starts with $\Theta(0) = (0, 0, \cdots, 0)$. It keeps track of $n$ independent Poisson clocks $\C_1, \C_2, \cdots, \C_n$, where the rate of $\C_j$ is $y_j$, and whenever a clock $C_j$ sends a signal throughout the time $0 \leq t \leq 1$, the process substitutes $\Theta$ with $\Theta^{\uparrow j}$ (see Eq.~(\ref{eq:magic}) for the definition of the latter). \emph{If} we were given an appropriate vector $y$ (for instance, the $y$ that corresponds to the probability of each element being picked by the optimum adaptive policy), then we could simulate this process by \textbf{\emph{discretizing}} the time (we will discuss discretizing later in detail).  However, we do not know any such vector a priori. The idea of our algorithm is to use a \textbf{\emph{greedy}} approach to select the entries that are going to be updated at any given point in time. See Figure~\ref{fig:generalizedcontgreedy} for the description of our algorithm.

Our method is a generalization of the continuous greedy approach of~\cite{CalinescuCPV08} performed on function $F$ as defined in Eq.~(\ref{expancion}). In fact, for the deterministic problem (i.e., when each $X_j$, if picked, will always have a realized value of $1$) our method becomes the equivalent of the continuous greedy method of~\cite{CalinescuCPV08}.  


\textbf{\textsc{The High Level Idea:}} Our algorithm (whose description can be seen below) keeps track of a vector $y(t)$, starting from $y(0) = (0, 0, \cdots, 0)$ (\textbf{see step (i)}). It then updates $y(t)$ by taking steps of size $\delta$ (for some carefully chosen $\delta$ as we will discuss later) throughout the time, until $t = 1$. In order to find out how to update $y(t)$, our algorithm first calculates the potential marginal contribution of each $y_j$ to the current $F(y(t))$. In order to mimic the function $F$ as defined in Eq.~(\ref{expancion}), our algorithm produces a randomly selected set $R(t)$ that contains each element $X_j$ with probability $y_j(t)$ independently at random. Then it simulates a realization of $R(t)$, namely $\Theta_{R(t)}$, using the probability distributions $g_j$. The marginal contribution of the element $y_j$ to $F(y(t))$ for this specific realization can be simply measured by $f(\Theta_{R(t)}^{\uparrow j}) - f(\Theta_{R(t)})$. In order to estimate the expected marginal contribution of $y_j$, our algorithm takes the average of the mentioned subtraction over a large enough number of samples (\textbf{see step (ii)}). Finally, our algorithm finds a maximum-weight independent set according to the weights defined as the expected marginal contribution of each element. This can be done via the classic greedy algorithm for finding the maximum-weight independent sets in matroids (see~\citealt{rado, gale, edmonds}). 
Finally, $y(t+\delta)$ is obtained by adding $\delta$ to the entries of $y(t)$ corresponding to the elements of the maximum-weight independent set (\textbf{see step (iii)}).

\begin{figure}[!t] 
\caption{The Stochastic Continuous Greedy Algorithm}
\begin{center}
\fbox{\parbox{6in}{ \vspace*{3mm} 
\begin{quote}
\begin{enumerate}[i.]
\item Let $\delta = 1/(\lceil\frac{3d}{\epsilon}\rceil)$, where $d$ denotes the rank of the matroid $\M = (\A, \I)$; $n = |\A|$, and $\epsilon$ is the discretization parameter.
Start from $t = 0$ and $y(0) = 0$.\newline 

\item 
Let $R(t)$ be a set containing each element $X_j \in \A$ independently with probability $y_j(t)$.  \\
Let $\Theta_{R(t)}$ be the realization of set $R(t)$, i.e., a vector whose $j$-th entry is independently drawn from the probability distribution function $g_j$ if $X_j \in R(t)$, and is $0$ if $X_j \notin R(t)$. \\
Let $w_j(t)$ be an estimate of $\E_{\Theta, R}\left[f(\Theta_{R(t)}^{\uparrow j}) - f(\Theta_{R(t)})\right]$, obtained by averaging over $\frac{4}{\delta^2}(1+\ln n - 0.5\ln \delta)$ samples of $\Theta_{R(T)}$. 
\medskip

\item Let $I(t)$ be a maximum-weight independent set in $\M$ according to the weights $w_j(t)$. Let $y(t+\delta) = y(t) + \delta~.~\1_{I(t)}.$ \newline
\item Increment $t:= t+\delta$; if $t < 1$, go back to (ii). Otherwise, return $y(1)$.
\end{enumerate}
\end{quote}
\medskip
}}
\end{center}
\label{fig:generalizedcontgreedy}
\end{figure}


In order to analyze the stochastic continuous greedy algorithm, we follow a  similar path to that of the proof of Theorem~\ref{thm:gap} in the previous section. We show that our sampling in step (ii) provides a good approximation of the weights, and in particular, the weight of the independent set selected at step (iii) in each iteration is close to the weight of the actual maximum-weight independent set. For the rest of this section, ``with high probability'' means with probability of at least $(1 - 1/\textrm{poly}(n))$.
The proofs of the following lemmas can be found in the online appendix. 

\begin{lemma}\label{superkhafan}
If the set $I(t)$ is chosen by step (iii) of the stochastic continuous greedy algorithm at time $t$, then with high probability, for any $t$ we have
\begin{equation*}
\sum_{j: X_j \in I(t)} \E_{\Theta, R}\left[f(\Theta_{R(t)}^{\uparrow j}) - f(\Theta_{R(t)})\right] \geq 
\left(\max_{I \in \I} \sum_{j: X_j \in I} \E_{\Theta, R}\left[f(\Theta_{R(t)}^{\uparrow j}) - f(\Theta_{R(t)})\right]\right)- 2d\delta~.~\OPT,
\end{equation*}
where $\OPT$ is the value of the optimum adaptive policy.
\end{lemma}

Now, we focus on the change in the value of $F(y(t+\delta))$ compared to that of $F(y(t))$. We show that the actual difference between $F(y(t+\delta))$ and $F(y(t))$ is bounded from below by $\delta (1 - d \delta)$ times the weight of the maximum independent set selected by step (iii). We note that this lemma could be seen as analogous to Eq.~(\ref{secondneq}) in our continuous Poisson process.\footnote{We remark that the proof of Lemma~\ref{updatebound} does not depend on the accuracy of the estimation and only uses a stochastic dominance result.}
\begin{lemma}\label{updatebound}
At each $t$ we have $$F(y(t+\delta)) - F(y(t)) \geq \delta(1-d\delta)\sum_{j: X_j \in I(t)} \E_{\Theta, R}\left[f(\Theta_{R(t)}^{\uparrow j}) - f(\Theta_{R(t)})\right],$$ where $d$ is the rank of the matroid $\M$.
\end{lemma}

The final step in the analysis of our algorithm is to provide an analogous statement to the differential equation of Eq.~(\ref{eq:derivative}). This is done through the following lemma whose proof is presented in the appendix.
\begin{lemma}
\label{lem:final}
Throughout the algorithm, with high probability, for every $t$ we have the following:
$$F(y(t + \delta)) - F(y(t)) \geq \delta (1 -  3d\delta) \sup_{y \in B(\M)}\{f^*(y)\} - F(y(t)).$$
\end{lemma}


Now we are ready to finish the analysis of the algorithm.
Let us first denote the value of $(1 -  3d\delta) \sup_{y \in B(\M)}\{f^*(y)\}$ by $\S$.  
From Lemma~\ref{lem:final} we have $\S - F(y(t+\delta)) \leq (1 - \delta)(\S - F(y(t)))$ with high probability throughout the algorithm. By induction and considering the fact that $F(y(0)) = 0$, we get $\S - F(y(k\delta)) \leq (1 - \delta)^{k}\S$ for every $k$ such that $k\delta \leq 1$. In particular, for $k = 1/\delta$ we will have
$$\S - F(y(1)) \leq (1 - \delta)^{1/\delta}\S \leq \frac{1}{e}\S.$$

Thus, $F(y(1)) \geq (1 - 1/e) \S \geq (1 - 1/e - 3d\delta)  \sup_{y \in B(\M)}\{f^*(y)\}$. Note that $\delta = 1/(\lceil\frac{3d}{\epsilon}\rceil)$. Hence, the point $y(1) \in B(\M)$ found by the stochastic continuous greedy algorithm with high probability has the value of at least $(1 - 1/e - \epsilon)\sup_{y \in B(\M)}\{f^*(y)\}$. Therefore, we obtain the following result:
\begin{lemma} \label{lemma:poly}
For any $\epsilon >0$, the stochastic continuous greedy algorithm finds a vector $y^* \in \B(\M)$ in polynomial time in $|\A|$ and $1/\epsilon$ such that $F(y^*) \geq  \sup_{y \in \B(\M)} (1-1/e-\epsilon)f^*(y)$.
\end{lemma}

{By Lemma~\ref{lem:fractional_non_adapt} and using vector $y^*$ mentioned in the above lemma, we can find in polynomial time a subset $S^* \in \I$ such that $F(S^*) \geq F(y^*) \geq \sup_{y \in \B(\M)} (1-1/e-\epsilon)f^*(y)$. Recall that by Lemma~\ref{lem:fplusstar}, we have $f^*(y) \geq f^+(y)$ for every $y \in \B(\M)$. Hence, $F(S^*) \geq (1 - 1/e - \epsilon) \sup_{y \in \B(\M)} f^+(y)$. On the other hand, by Lemma~\ref{adaptiveupperbound}, $\sup_{y \in \B(\M)} f^+(y)$ is an upper bound on the value of the optimal adaptive policy. Therefore, we can conclude the following.}



{
\begin{theorem}[Approximately Optimal Polynomial-Time Adaptive Policy] \label{thm:poly}
For any $\epsilon >0$ and any instance of an SMSM problem, a non-adaptive policy that obtains a $(1-1/e-\epsilon)$-fraction of the value of the optimal adaptive policy
can be found in polynomial time $($with respect to $|\A|$ and $1/\epsilon)$.
\end{theorem}
}


%% file: matroidpolicies.tex
\section{Approximation Ratio of Simple Myopic Policies} \label{sec:matroidpolicies}
In this section, we present an adaptive myopic policy for the SMSM problem with a general matroid constraint. The set of feasible solutions is the intersection of $\kappa$ matroids. Let  $\M_1 = (\A, \I_1), \M_2=(\A, \I_2), \cdots, \M_{\kappa}=(\A, \I_{\kappa})$ be $\kappa$ matroids all defined on the ground set of elements $\A$. The feasible set for the SMSM problem that we study in this section is  $\I = \I_1 \cap \I_2 \cap \cdots \cap \I_{\kappa}$. We will present a myopic policy whose approximation ratio with respect to the optimal adaptive policy is $\frac{1}{\kappa+1}$.
At the end of this section, we explain that the myopic policy achieves the approximation ratio of $\efac$, if $\kappa = 1$ and the matroid is uniform.

The policy is given in Figure~\ref{fig:adaptivepolicy}.
Remember that the feasible set of elements is given by $\I$.
At each iteration, from the elements in $\A$ that have not been picked yet,
the policy chooses an element with the maximum expected marginal value.
We denote by $S_t$ the set of elements chosen by the adaptive policy up to iteration $t$.
By abuse of notation, let $s_t$ denote the realizations of {all} of the elements in $S_t$.
Also, $U_t$ is the set of elements probed but not chosen by the policy due to the matroid constraint.

\begin{figure}[!t]
\caption{The Adaptive Myopic Policy}
\begin{center}
\fbox{\parbox{5.5in}{ \vspace*{3mm} 
\label{fig:adaptivepolicy}
%
{
\begin{description}
\item\i Initialize $t = 0, S_{0} = \emptyset, U_{0} = \emptyset$, $s_0 = \emptyset$.
\item\i While $(U_{t} \cup S_{t} \neq \A)$ do
\begin{description}
\item\i\i $t \leftarrow t + 1$; $S_{t} \leftarrow S_{t-1}$; $U_{t} \leftarrow U_{t-1}$;  $s_{t} \leftarrow s_{t-1}$.
\item\i\i Repeat
\item\i\i\i   Find $X_i\in\argmax_{X_i \in \A \setminus (U_{t} \cup S_{t})} \E[F(S_{t-1}\cup \{X_i\})|s_{t-1}]$.
\item\i\i\i   If $S_{t-1} \cup \{X_i\} \notin \I$, then
\begin{description}
\item\i\i\i\i $U_{t} \leftarrow U_{t} \cup \{X_i\}$,
\end{description}
\item\i\i\i   else
\begin{description}
\item\i\i\i\i $S_{t} \leftarrow S_{t} \cup \{X_i\}$.
\item\i\i\i\i Observe $x_i$ and update the state $s_t$, i.e., $s_t\leftarrow s_{t-1}\cup \{x_i\}$.
\end{description}
\item\i\i Until $(S_t \neq S_{t-1})$ or $(U_{t} \cup S_{t} = \A)$.
\item\i End (while)
\end{description}
\end{description}
}
\medskip
}}
\end{center}
\medskip
\end{figure}

Note the exit condition for the repeat loop. The repeat loop will continue until either a new element $X_i$ is added to the set of selected elements (i.e., $S_t = S_{t-1} \cup \{X_i\} \neq S_{t-1}$) or all elements in the ground set $\A$ are probed (i.e., $U_{t} \cup S_{t} = \A$).  In the former, the policy will continue probing more elements in order to possibly add new elements to $S_t$. In the latter, however, after exiting the repeat loop, the policy immediately jumps out of the while loop, too. Consequently, the policy ends.

\begin{theorem}[Adaptive Myopic Policy] \label{halfapprox}
If feasible domain $\I$ is the intersection of $\kappa$ matroids, then the approximation ratio of the adaptive myopic policy with respect to any optimal adaptive policy is $\frac{1}{\kappa+1}$.
\end{theorem}

The proof is given in Appendix~\ref{app:matroidpolicies}. In order to prove the above bound, we will rely on the following lemma proved by~\cite{FisherNW78-2}. Note that in each iteration of the while loop, at least one new element of $\A$ is probed. Hence, $|\A|$ is an upper bound on the number of iterations of the while loop  in the myopic policy. For the clarity of the presentation, if the policy ends at some iteration $t^*$ where $t^* < {|\A|}$, we define $U_{i} := U_{t^*}$ and $S_{i} := S_{t^*}$ for all $i > t^*$. (In other words, we assume that the policy continues until the $|\A|$-th iteration but does not do anything after iteration $t^*$.) 
%
%
\begin{lemma}\label{lem:fnwp1} {\em (\citealp{FisherNW78-2})}
{Let $P \in \I_1\cap \I_2 \cap \cdots \cap \I_{\kappa}$ be any arbitrary subset in the intersection of $\kappa$ matroids, $\M_1, \M_2, \cdots, \M_{\kappa}$. For $0 \leq i \leq {|\A|}$, let $C_{i} = P \cap (U_{i+1} \setminus U_{i})$, where $U_i$'s are coming from the Adaptive Myopic Policy. We will have $\sum_{i=0}^{t} |C_i| \leq \kappa t$, for every $1 \leq t \leq {|\A|}$. }
\end{lemma}

Note that $C_i$'s, $U_i$'s, and $S_i$'s are random sets that depend on the decisions made by the policy and possibly based on the outcomes of the previous realizations of the chosen elements. However, the lemma holds for every realization path because the above property is a consequence of
the matroid constraint, not the realizations of the element chosen by the policy.

Define $\Delta_t := F(S_{t}) - F(S_{t-1})$, for $1 \leq t \leq {|\A|}$.
The main difficulty in proving the theorem compared to the non-stochastic case is that
the realized values of $\Delta_t$ are not always decreasing --- note that $\E[\Delta_t|s_{t}] \geq \E[\Delta_{t+1}|s_{t}]$ does not necessarily hold due to the stochastic nature of the problem. (Because at time $t$ the actual value of $\Delta_t$ has been realized, and it may have turned out to be very small. On the contrary, $\E[\Delta_t|s_{t - 1}] \geq \E[\Delta_{t+1}|s_{t - 1}]$ always holds, and this is what we are going to use in our proof.)
In addition, the sequence of elements chosen by the optimal adaptive policy is random; see the proof in the appendix for more details. \upt{We point out that ${1\over 1+\kappa}$ is a tight bound on the approximation ratio of the myopic policy~\citep{FisherNW78-2}.}
%
%
%
%
%
%
In Appendix~\ref{app:matroidpolicies}, we show that the myopic policy obtains a stronger approximation ratio if the constraint matroid is uniform.

\begin{theorem}[Uniform Matroid] \label{thm:uniform} 
Over a uniform matroid, the approximation ratio of the adaptive myopic policy 
with respect to the optimal adaptive policy is equal to $\efac$.
\end{theorem}

\upt{We note that achieving any approximation ratio better than the above is computationally intractable. In fact, as mentioned earlier, even in deterministic settings the problem of maximizing a submodular function with respect to cardinality constraints cannot be approximated with a ratio better than $\efac$ within polynomial time; see~\cite{Feige98}.}



%
%
%
%

%% file: discussion.tex
\section{Discussions} \label{sec:discussion}
In this section, we discuss some of our previous assumptions and explain how our results would change in their absence.

\subsection{Computation of $F$}  \label{sec:f_computation}
The problem of calculating $F$ and functions similar to it has appeared in several related works and is usually solved via sampling. For instance, in the framework of~\cite{VondrakThesis07, Vondrak08}, where the author deals with a similar notion of $F$ but only with Bernoulli distribution for all distributions $g_i$-s, it is shown that with repeated sampling one can get an estimation of $F$ within a multiplicative error of $1\pm 1/\textrm{poly}(n)$. The following proposition generalizes their results to Lipschitz continuous functions.
\begin{proposition}[Additive Error] \label{lem:sampling}
Suppose $f$ is $K$-Lipschitz continuous and for all $i$, $\emph{\Var}[X_i] \leq \V$. Then, for any arbitrary values of $0 < \epsilon < 1$ and $\delta > 0$, the average value of $\ff(\Theta_S)$ for $t =  \lceil K^2 n^2 \V \epsilon^{-1} \delta^{-0.5} \rceil$ independent samples is within the interval $\Big(\emph{\E}[\ff(\Theta_S)] - \epsilon, \emph{\E}[\ff(\Theta_S)] + \epsilon \Big)$, with probability at least $1-\delta$. 
\end{proposition}
The proof is given in the online appendix. The above lemma could be used to achieve any desirable polynomially small additive error. For instance, $\lceil K^2 n^5 \V \rceil$ samples is enough to ensure that with probability $1 - n^{-2}$, the average sampled value of $f(\Theta_S)$ is within the interval $\Big(\E[\ff(\Theta_S)] - n^{-2}, \E[\ff(\Theta_S)] + n^{-2} \Big)$. 

Note that if we have an upper bound on the values of $f$ (for instance, 
where $X$'s can only take binary values), then we can use Chernoff inequality in the proof. That allows us to get the same bounds by only $O(\log K + \log n + \log \V)$ samples. Also, since in this case the values $F$ can take would be bounded, one can easily come up with multiplicative error terms. This is formalized via the following proposition (proved in the online appendix).

\begin{proposition}[Multiplicative Error] \label{lem:sampling2}
Suppose for any realization $\Theta_S$ with $\emph{\Pr}[\Theta_S] > 0$, we have $1 \leq f(\Theta_S) \leq \FF$. Then, for any arbitrary values of $0 < \epsilon < 1$ and $0<\delta$, the average value of $f(\Theta_S)$ for $t = \lceil-2\FF^2 \ln(\delta/2) \epsilon^{-1}\rceil$ independent samples is within the interval $\Big(\emph{\E}[f(\Theta_S)](1-\epsilon), \emph{\E}[f(\Theta_S)](1+\epsilon)\Big)$, with probability at least $1 - \delta$.
\end{proposition}

Throughout this paper we have assumed that the values of the form $\E[f(\Theta_S)]$, or equivalently $F(S)$, for every subset $S$ are accessible via an oracle. However, if such an oracle is not available, then one can incorporate the above lemmas to achieve our results with the desirable degree of precision. To observe this, we note that our algorithms have at most $n$ steps, and at each step there are at most $n$ values of the form $\E[f(\Theta_S)]$ that should be estimated. Hence, at most $n^2$ estimations of function $\E[f(\Theta_S)]$ will be needed. Due to Proposition~\ref{lem:sampling}, each of these estimations will have at most $\epsilon$ additive error with probability at least $1 - \delta$ if the average is taken over $\lceil K^2 n^2 \V \epsilon^{-1} \delta^{-0.5} \rceil$ samples. Therefore, union bound implies that with a probability at least $1- n^2\delta$, all these samples will be within $\epsilon$ error from their true value. Since all the inequalities in our proofs have at most $n$ terms, an error of at most $\epsilon n$ will appear in our final results. Hence, if we set $\delta = n^{-3}$, we need to average over  $\lceil K^2 n^2 \V (\tfrac{\epsilon}{n})^{-1}  (n^{-3})^{-0.5} \rceil$ (or simply, $\lceil K^2 n^{4.5} \V {\epsilon}^{-1} \rceil$) samples throughout our algorithms in order  to ensure that with probability at most $1 - n^2 . n^{-3}$ (or simply, $1 - 1/n$) our results will stay valid with an additional error of $\epsilon$.

\subsection{Choosing Elements with Replacement} \label{sec:mult}

\upt{So far, in our model we did not allow for multiple draws of any element. In other words, once an element is picked and its value is realized, it cannot be picked again later. 
However, this assumption is made without loss of generality and in order to simplify the presentation. 
This follows from the observation that one can reduce the problem in which replacements are allowed to our original setting by creating identical and independent copies of each random variable in order to simulate multiplicity.}

\upt{Consider the problem of maximizing a stochastic monotone submodular function defined on a function $f:\R_+^n \rightarrow \R_+$. 
Also, suppose that we allow for multiple selections of the entries, and if an entry corresponding to any $X_i$ is selected multiple times, then its final value will be the \emph{maximum} of all the realized values in the draws.\footnote{This is in line with the current literature on selection problems with ``multiplicity''. For instance, see~\cite{GoemansV06}, where multiple selection of a set in the stochastic set cover problem is considered.}
More formally, we allow for $m$ copies of each $X_i$, where $m$ is an upperbound on the maximum number of elements that could be picked. (For instance, with matroid constraints, $m$ could be the rank of the matroid.) Then, we construct $m$ copies of each $X_i$, namely, $X^{(i)}_1, X^{(i)}_2, \cdots, X^{(i)}_m$. Let $a_i$ denote the number of times element $X_i$ is selected, and $x^{(i)}_1, x^{(i)}_2, \cdots,  x^{(i)}_{a_i}$ represent the corresponding realizations. As before, the value of all unrealized elements $x^{(i)}_k$, $a_i + 1\le k\le m$, is defined to be zero. The final value of the objective function will be denoted by  function $\phi:\R_+^{mn} \rightarrow \R_+$ which is defined as follows. 
\begin{equation}
\label{eq:replacement}
\phi(x^{(1)}_1, \cdots,  x^{(1)}_{m}, x^{(2)}_1, \cdots,  x^{(2)}_{m}, \cdots, x^{(n)}_1, \cdots,  x^{(n)}_{m}) :=  f\left(\max_k \{x^{(1)}_k\}, \max_k \{x^{(2)}_k\}, {\cdots}, \max_k \{x^{(n)}_k\}\right).
\end{equation}
The next lemma shows that $\phi$ is monotone and submodular. The proof is presented in the 
appendix.\begin{lemma} \label{lem:replacement}
If function $f:\R_+^n \rightarrow \R_+$ is monotone submodular, then function $\phi:\R_+^{mn} \rightarrow \R_+$, as defined in {\em Eq.~(}\ref{eq:replacement}{\em )}, will be monotone submodular.
\end{lemma}
}
The above lemma shows that stochastic monotone submodular optimization problems with replacement, i.e., when multiple selection of an element might change its value, can be reduced to our original setting in which replacement is not allowed. Hence, our approximation guarantees and the adaptivity gap
will still hold if we allow for multiple draws. Our tightness results for approximation guarantees come from deterministic examples. Therefore, for the same reason, they hold as well when selection with replacement is allowed. 

%
%
We note that the above observations are in sharp contrast with minimization problems (e.g., minimum set cover) and threshold problems in which the goal is to reach a value of function above a certain hard-constrained threshold. For instance, in the problem of stochastic set cover  (where the objective is to cover \emph{all} the elements of a ground set using the minimum number of a given collection of its subsets), allowing a policy to choose multiple copies of an element significantly reduces the adaptivity gap, as shown by~\cite{GoemansV06}.

\subsection{Large Adaptivity Gap without the Independence Assumption} \label{apx:unbouned}
In this section, we present an example that shows that the adaptivity gap without the independence assumption can be at least $n/2$. 
Consider an instance of {SMSM2($n, M$)} with the ground set $\A = \{X_0, X_1, \cdots, X_n\}$, where $X_0$, with equal probability of ${1\over n}$,  takes one of the values in $\{1,2,\ldots,n\}$. Also, for $i$, $1\le i\le n$, we have $X_i = M \times \textbf{I}\{X_0 = i\}$ for some $M \gg n$, i.e., $X_j = M$ if $X_0 = j$; otherwise, $X_j=0$.
The goal is to maximize a \emph{linear} function $f$ that corresponds to the sum of the realized values of the chosen elements.
The constraint is to select at most two elements, i.e., an independent set from the matroid $\M = (\A, \{S \subseteq \A, |S| \leq 2\})$.

It is easy to see that any non-adaptive policy would yield a value of at most $2M/n$ because each element selected will contribute at most $M/n$ in expectation. However, an adaptive policy can select $X_0$ first and then select $j=x_0$, and therefore, obtain a value of at least $M$ (exactly $M+j$ to be more precise). Hence, the adaptivity gap is at least $n/2$.

\upt{Recently, \cite{GolovinK11,GolovinK11_2} show that the myopic adaptive policy provides an $\tfrac{e-1}{e}$-approximation to the optimal adaptive policy under a cardinality constraint when the {\em adaptive submodularity} assumption holds.
 Under adaptive submodularity, the random variables of interest (namely, $X_1, X_2, \cdots, X_n$) are not necessarily independent. Instead, it suffices that the function $f$ satisfies diminishing conditional expected marginal returns over partial realizations. In addition, they show an approximation factor of $\tfrac{1}{\kappa+1}$ for myopic adaptive policies under the intersection of $\kappa$ matroids. Their results show that our Theorems~\ref{halfapprox} and \ref{thm:uniform} will also hold in an adaptive submodular setting. In addition, see~\cite{AgrawalDSY12} who introduce the notion of {\em price of correlations} and study the effect of ignoring correlations in optimization problems.}

\upt{\subsection{Large Adaptivity Gap without the Monotonicity Assumption} \label{apx:nonmonotone}
Another important assumption in our model is the monotonicity of the objective function. This assumption is necessary in order to obtain our adaptivity gap results. Here, we provide an example of non-monotone stochastic submodular optimization with arbitrarily large adaptivity gap. Our example is an instance of the stochastic maximum $k$-coverage (similar to Section~\ref{sec:tightexample}), with an additional probing cost appearing in the objective. Formally, we are given $\A = \{X_1, X_2, \cdots, X_n\}$, a collection of finite sets, where a cost $c_i$ is associated with any $X_i$. The goal is as follows.
$$\max_{S \subseteq \{1, 2, \cdots, n\}, |S| \leq k} |\bigcup_{i \in S} X_i| - \sum_{i \in S} c_i.$$ We note that $|\cup_{i \in S} X_i|$ is a submodular function of $S$, and $\sum_{i \in S} c_i$ is an additive function of $S$. Hence, the objective is submodular (but not necessarily monotone) in $S$.}

\upt{In the stochastic version of the problem, each $X_i$ -- if chosen -- will cover a set $A_i$ with probability $p$ independently at random, or will turn out to be empty with probability $1-p$. We now formally show that this is a special case of our setting. We work with the state space $(\theta_1, \theta_2, \cdots, \theta_n)$ by starting from $(\nill, \nill, \cdots, \nill)$. Once an element $X_i$ is chosen, if it turns out to be non-empty, then $\theta_i := 1$, otherwise, $\theta_i := -1$. We define the value function $f:\{-1, \nill, 1\}^n \rightarrow \R$ as follows.
\begin{equation}
\label{eq:nejatbakhsh}
f(\theta_1, \theta_2, \cdots, \theta_n) := \Big\vert\bigcup_{\theta_i = 1} A_i\Big\vert - \sum_{i: \theta_i \neq \nill} c_i.
\end{equation}
Similar to the definition of the value function (see Eq.~(\ref{eq:value})), we replace $\nill$ with $0$ whenever we calculate the value of function $f$. The proofs of Lemmas~\ref{lem:nejatbakhsh}, \ref{lem:nonadaptivenonmonotone}, and \ref{lem:adaptivenonmonotone} are presented in the appendix.%
\begin{lemma}
\label{lem:nejatbakhsh}
The function $f$ defined by~(\ref{eq:nejatbakhsh}) is submodular.
\end{lemma}
Now, we consider the problem of SMSM3($n, c, p$) in which all $A_i$-s are equal to $\{1\}$, all $c_i$-s are equal to a fixed $c := -(1-p)\ln (1-p)$, and also $k = n$ (meaning any subset $S \subseteq \{1, 2, \cdots, n\}$ can be chosen).
\begin{lemma} \label{lem:nonadaptivenonmonotone}
The optimal non-adaptive policy for SMSM3($n, c, p$) is to pick a subset $S$ consisting of only one element. The expected value of  this policy is $p-c$. 
\end{lemma}
Now, consider an adaptive policy $\pi$ that picks the elements in $\{1, 2, \cdots, n\}$ one by one, until the first time that an $X_i$ turns out to be non-empty, and hence, covers $\{1\}$. Then the policy stops.
\begin{lemma} \label{lem:adaptivenonmonotone}
The expected value of the adaptive policy $\pi$ is at least $\left(1 - (1-p)^n\right) - c/p$.
\end{lemma}
By combining the lemmas above, we observe that the adaptivity gap of SMSM3($n, c, p$) is at least as large as:
$$\frac{\left(1 - (1-p)^n\right) - c/p}{p - c}=\frac{p-c - p(1-p)^n}{p(p - c)}={1\over p} - \frac{(1-p)^n}{(p - c)} .$$
Note that fixing $p$, by choosing large enough values of $n$, the bound gets arbitrarily close to $1/p$. Hence, taking the value of $p$ sufficiently close to $0$ makes the adaptivity gap arbitrarily large.
}

\upt{Note that in our example, the value function $f$ takes negative values at some points. Our previous results were derived under the assumption that $f$ is defined on $R_+^n$ and only takes non-negative values. This is without loss of generality for monotone functions if we assume the value of taking ``no action'' is zero (i.e., $f(\nill, \cdots, \nill) = 0$). We leave it as an open question to find the adaptivity gap for non-monotone non-negative stochastic submodular maximization.}


%% file: Conclusion.tex
\section{Concluding Remarks} \label{sec:conclusion}
In this paper we studied the problem of maximizing monotone submodular functions with respect to matroid constraints in a stochastic setting. Our model can be applied to various problems that involve both diminishing marginal returns and a stochastic environment. In order to capture the effect of partial contributions, we considered real-valued submodular functions instead of submodular set functions.
We showed that a myopic adaptive policy is guaranteed to achieve a ($1 - \tfrac{1}{e} \approx 0.63$)-approximation of the optimal (adaptive) policy. Also, we studied the concept of the adaptivity gap in order to compare the performance of non-adaptive policies (which are very easy to implement) to that of adaptive policies. We showed that a very straightforward myopic non-adaptive policy achieves a ($\tfrac{1}{2} \times (1 - \tfrac{1}{e}) \approx 0.316$)-approximation of the value of the optimal policy. Moreover, we showed that the adaptivity gap for these problems is at most $\tfrac{e}{e-1}$, which implies the existence of a non-adaptive policy that achieves a 0.63-approximation of the optimal policy. Finally, we provided a polynomial algorithm to find a non-adaptive policy that achieves a $(0.63 - \epsilon)$-approximation of the optimal policy for any positive  $\epsilon$. In our proofs, we generalized the techniques from the previous works in the literature (especially those in~\citealt{Vondrak08} and \citealt{CalinescuCPV08}) to the stochastic and continuous setting by proving new stochastic dominance results that could be of independent interest.

There are many interesting questions that we leave open for future research. One such question is whether our algorithms perform better for the specific submodular functions that appear in practice. For instance, the submodular functions that arise from the context of viral marketing have a very specific structure due to the properties of the underlying graph of relationships in social networks (such as being self-similar and scale-free, see~\citealt{Barabasi99}). Exploiting these properties may lead to improved analysis of our policies or to new policies with better performance guarantees. Also, studying similar classes of functions for which the adaptivity gap is bounded would be an interesting extension. We believe that one promising class of such functions could be the almost-concave functions, which are the functions whose values can be approximated closely by a concave function.

\upt{Another interesting direction would be to study the problem of maximizing {\em non-monotone} submodular functions in a stochastic setting; applications include facility location problem and maximum entropy sampling (see \citealt{LeeMNS}). The example provided in Section~\ref{apx:nonmonotone} showed that the adaptivity gap of non-monotone stochastic submodular maximization can be unbounded if we do not assume the non-negativity of the objective function. We leave it as an open question to find the adaptivity gap for \emph{non-monotone non-negative} stochastic submodular maximization. 
We emphasize that finding provably good approximate solutions for this problem has turned out to be more difficult than that of the monotone case even in the deterministic setting. Various methodologies have been used to approach the problem including local search~\citep{LeeMNS,Vondrak-nonmonotone} and simulated annealing~\citep{GharanV}. \cite{FeldmanNS} were the first to extend the \emph{continuous greedy} approach of \cite{CalinescuCPV08} to the non-monotone setting, and provided a $1/e$-approximation for the problem under cardinality constraint. Recently, \cite{BuchbinderFNS14} built upon the \emph{double greedy} approach\footnote{The double greedy approach was introduced by \citealt{BuchbinderFNS12} for unconstrained non-monotone submodular optimizations. Their technique, generalizes upon the continuous greedy approach by keeping track of two solutions (starting from $(0, 0, \cdots, 0)$ and $(1, 1, \cdots, 1)$, instead of only one solution starting from $(0, 0, \cdots, 0)$ as in the continuous greedy) and continuously moving them towards each other by a greedy rule until they coincide at time $t = 1$.} and used randomization to deal with non-monotonicity of the objective function. A promising approach to extend their results to the stochastic setting would be to combine it with our stochastic continuous greedy and develop a unified \emph{stochastic continuous double greedy} approach that could handle both the stochasticity of the environment and the non-monotonicity of the objective function.}

%% file: appendix.tex
\section{Proofs} \label{app:proof}

\subsection{Appendix to Section~\ref{sec:gap}} \label{app:gap}

\begin{proof}{Proof of Lemma~\ref{adaptiveupperbound}}
An optimal adaptive policy only chooses independent sets of $\M$. Due to monotonicity,
all of these are also the bases of the matroid. Hence, an optimal adaptive policy ends up selecting one of the basis of the matroid. Thus, vector $y$, whose $i$-th entry represents the probability that element $X_i\in\A$ is ultimately chosen by this policy, is a convex combination of the characteristic vectors of the basis of $\M$. Moreover, the expected value of the adaptive policy is bounded by $f^{+}(y)$
because the policy has to satisfy the three properties mentioned earlier.
\qed\end{proof}
\medskip

\begin{lemma} \label{lem:fplusstar}
For $y\in\B(\M)$, we have $f^+(y)\le f^*(y)$.
\end{lemma}
\begin{proof}{Proof}
Consider the function $f^+$. Following the notation of $\Theta^{\uparrow j}$, for every $x \in \R$ we define $\Theta^{\uparrow j: x}$ as a vector that has all entries except the $j$-th one as the same as $\Theta$, and its $j$-th entry is defined as $\max\{\theta_j, x\}$.  Fix any $y \in \B(\M)$, any feasible probability measure $\alpha$ with respect to $y$ (see definition~(\ref{extensionintegral}) for $f^+$), and any given vector $\Thetazir \in \R_+^n$. We have
\begin{eqnarray}
\nonumber \int_{\Theta} \alpha_{\Theta} \ff(\Theta) d \Theta &\leq&
\int_{\Theta} \alpha_{\Theta}\left[\ff(\Thetazir) + \sum_{j = 1}^n \Big(\ff(\Thetazir^{\uparrow j, \theta_j}) - \ff(\Thetazir)\Big)\right] d \Theta \\
\nonumber&=& \ff(\Thetazir) + \sum_{j = 1}^n y_j \left[ \int_{\theta_j > \thetazir_j} \Big(\ff(\Thetazir^{\uparrow j, \theta_j}) - \ff(\Thetazir)\Big) g_j(\theta_j)  d\theta_j\right].
\end{eqnarray}

The inequality above holds due to the submodularity of $f$, and the equality is a consequence of the properties of the feasible probability measure $\alpha$. However, $\int_{\theta_j > \thetazir_j} \ff(\Thetazir^{\uparrow j, \theta_j}) g_j(\theta_j)  d\theta_j$ is
simply $\E[f(\Thetazir^{\uparrow j})]$. Hence,

$$\int_{\Theta} \alpha_{\Theta} \ff(\Theta) \leq \ff(\Thetazir) + \sum_{j = 1}^{n} y_j \Big(\E\left[f(\Thetazir^{\uparrow j})\right] - f(\Thetazir)\Big).$$

Note that the inequality above holds for any arbitrary $y \in \B(\M)$, any feasible probability measure $\alpha$, and any vector $\Thetazir \in \R_+^n$. Thus, we have
\begin{equation}  \label{starplus}
f^+(y) = \sup_{\alpha}\left\{\int_{\Theta} \alpha_{\Theta} \ff(\Theta)d \Theta\right\} \leq \inf_{\Thetazir} \left\{\ff(\Thetazir) + \sum_{j = 1}^{n} y_j \Big(\E\left[f(\Thetazir^{\uparrow j})\right] - f(\Thetazir)\Big)\right\} = f^*(y) .
\end{equation}
\qed\end{proof}

\subsection{Appendix to Section~\ref{sec:matroidpolicies}} \label{app:matroidpolicies}
\medskip

\begin{proof}{Proof of Theorem~\ref{halfapprox}:}
Let $P$ be the (random) set of elements chosen by the optimal adaptive policy. Let $S$ denote the final selection of the elements by our myopic policy, i.e., $S = S_{|\A|}$. 
Moreover, let $t$ be any arbitrary number between $1$ and ${|\A|}$.
Consider a realization $s_t$ of $S_t$.
By submodularity and monotonicity of $F$ we have
\begin{eqnarray}
\nonumber \E[F(P) - F(S)|s_t] \leq \E[F(P \cup S) - F(S)|s_t] & \le & \E[\sum_{l\in P\setminus S} (F(S + l) - F(S))|s_t].
\end{eqnarray}

The expectations in the above inequality are taken over the probability distribution of all possible realizations conditioned on the realized values of elements in $S_t$  are equal to $s_t$.

Since the above inequality holds for all $s_t$, we have
\begin{eqnarray}
\label{eq:initial} \E[F(P)] - \E[F(S)] & \le & \E[\sum_{l\in P\setminus S} (F(S + l) - F(S))].
\end{eqnarray}

Now, define $C_t = P \cap (U_{t+1} \setminus U_t)$, for all $0 \leq t \leq {|\A|}$. In other words, $C_t$ represents the elements of the optimum solution $P$ that our myopic policy has probed (but not picked) after picking its $t$-th and before picking its $(t+1)$-st elements. By the construction of the policy, we have $U_0 \subseteq U_1 \subseteq \cdots \subseteq U_{|\A|}$. Hence, $(U_{t+1} \setminus U_t)$ and $(U_{t'+1} \setminus U_{t'})$ are  disjoint for any $t \neq t'$. Therefore, $C_t$'s are all disjoint. On the other hand, Lemma~\ref{lem:fnwp1} implies that $|C_0| = 0$ and hence $C_0 = \emptyset$. Therefore, $\bigcup_{t=1}^{{|\A|}} C_{t}$ represents all the elements in $P$ that were not selected by our policy, i.e., $\bigcup_{t=1}^{{|\A|}} C_{t} = P\setminus S$. As a result, we can rewrite Inequality~(\ref{eq:initial}) as follows:
\begin{eqnarray}
\nonumber \E[F(P)] - \E[F(S)] & \le & \sum_{t=1}^{{|\A|}} \E[\sum_{l\in C_t} (F(S + l) - F(S))].
\end{eqnarray}

By expanding the expectation we have
\begin{eqnarray}
\label{eq:adapt1}  \E[F(P)] - \E[F(S)] & \le & \sum_{t=1}^{{|\A|}} \int_{s_{t-1}: S_{t - 1}\in \I} \E[\sum_{l\in C_t} F(S + l) - F(S)|s_{t - 1}] \Pr[s_{t-1}] ds_{t-1}
\end{eqnarray}

Fix any $s_{t-1}$ and $l \in C_t$. Note that the myopic policy
probes the elements in decreasing order of their expected marginal value and it has picked the $t$-th element of $S$ (i.e., the only element in $S_{t} \setminus S_{t-1}$) before $l \in C_t$. Hence, we have 
$\E[F(S_{t - 1} + l) - F(S_{t-1})|s_{t-1}] \leq   \E[F(S_{t}) - F(S_{t-1})|s_{t-1}] = \E[\Delta_t|s_{t-1}].$
Now, since $F$ is submodular and $S_{t-1} \subseteq S$, for any $l\in C_{t}$ we have
$\E[\sum_{l\in C_t} F(S + l) - F(S)|s_{t-1}] \leq \E[\sum_{l \in C_t}\Delta_t|s_{t-1}].$
By plugging the above inequality into~(\ref{eq:adapt1}), we get
\begin{eqnarray} \label{eq:adapt2}
\nonumber \E[F(P)] - \E[F(S)] &\leq& \sum_{t=1}^{{|\A|}} \int_{s_{t - 1}: S_{t - 1}\in \I} \E[\sum_{l\in C_t} \Delta_t|s_{t - 1}] \Pr[s_{t-1}] ds_{t-1}.
\end{eqnarray}
Using telescopic sums and the linearity of expectation, we derive the following; here, $\Delta_{{|\A|}+1}$ is defined as $0$:
\begin{eqnarray}
\nonumber  \E[F(P)] - \E[F(S)] & \leq & \sum_{t=1}^{{|\A|}} \int_{s_{t-1}: S_{t - 1}\in \I} \E[\sum_{l\in C_t} \sum_{j = t}^{{|\A|}} (\Delta_j - \Delta_{j + 1})|s_{t - 1}] \Pr[s_{t-1}] ds_{t-1} \\
\nonumber &=& \sum_{j = 1}^{{|\A|}} \sum_{t=1}^{j} \int_{s_{t-1}: S_{t - 1}\in \I}
\E[\sum_{l\in C_t} (\Delta_j - \Delta_{j + 1})|s_{t - 1}] \Pr[s_{t-1}] ds_{t-1}.
\end{eqnarray}

{Note that by applying the law of total probability, for every $t$ and $j$ the integral term in the above is in fact equal to $\E[\sum_{l \in C_t} (\Delta_j - \Delta_{j + 1})]$.
Again, we use the law of total probability, but this time by conditioning on $s_{j - 1}$ instead of $s_{t - 1}$. We will have}
\begin{eqnarray}
\nonumber \E[F(P)] - \E[F(S)] & \leq & \sum_{j = 1}^{{|\A|}} \sum_{t=1}^{j}
\int_{s_{j-1}: S_{j - 1}\in \I} \E[\sum_{l\in C_t} (\Delta_j - \Delta_{j + 1})|s_{j-1}] \Pr[s_{j-1}] ds_{j-1}.
\end{eqnarray}
Note that the term $(\Delta_j - \Delta_{j-1})$ in the innermost summation does not depend on the index of the sum, i.e., $\l \in C_t$. Hence, the r.h.s. can be written as follows: 
\begin{eqnarray}
\nonumber  &=& \sum_{j = 1}^{{|\A|}} \sum_{t=1}^{j}
\int_{s_{j-1}: S_{j - 1}\in \I} \left(E\left[|C_t| \E[\Delta_j - \Delta_{j + 1}|s_{j-1}]|s_{j-1}\right]\right) \Pr[s_{j-1}] ds_{j-1}.
\end{eqnarray}
Note that conditioned on $s_{j-1}$, the term $\E[\Delta_j - \Delta_{j + 1}|s_{j-1}]$ is by definition a constant, and we can take it out from the outer expectation. Hence,
\begin{eqnarray}
\label{eq:kmatroid}
 \E[F(P)] - \E[F(S)] & \leq & \sum_{j = 1}^{{|\A|}}
\int_{s_{j-1}: S_{j - 1}\in \I} \left(\E[\sum_{t=1}^{j} |C_t||s_{j-1}] \E[\Delta_j - \Delta_{j + 1}|s_{j-1}]\right) \Pr[s_{j-1}] ds_{j-1}.
\end{eqnarray}
We now use Lemma~\ref{lem:fnwp1}, which implies that in every realization $\sum_{t=1}^{j} |C_t| \leq \kappa j$. We also use the fact that due to the submodularity and the rule of the policy, we have $\E[(\Delta_j - \Delta_{j + 1})|s_{j-1}] \geq 0$. We conclude that
{
\begin{eqnarray*}
\E[F(P)] - \E[F(S)] & \leq & \sum_{j = 1}^{{|\A|}}
\int_{s_{j-1}: S_{j - 1}\in \I} \left(\E[\sum_{t=1}^{j} |C_t||s_{j-1}] \E[\Delta_j - \Delta_{j + 1}|s_{j-1}]\right) \Pr[s_{j-1}] ds_{j-1}\\
&\leq& \sum_{j = 1}^{{|\A|}}
\int_{s_{j-1}: S_{j - 1}\in \I} \kappa j \E[(\Delta_j - \Delta_{j + 1})|s_{j-1}] \Pr[s_{j-1}] ds_{j-1} \\
 & = & \sum_{j = 1}^{{|\A|}}
\int_{s_{j-1}: S_{j - 1}\in \I}  \kappa\E[\Delta_j|s_{j-1}] \Pr[s_{j-1}] ds_{j-1} \\
 & = & \kappa\sum_{j = 1}^{{|\A|}}  \E[\Delta_j] = \kappa\E[F(S)].
\end{eqnarray*}
}
%
Therefore, $\E[F(P)] \le (\kappa+1) \E[F(S)]$, as desired.
Finally,~\cite{FisherNW78-2} have shown that even in the non-stochastic setting, in the worst case, the approximation ratio of the greedy algorithm (hence the myopic policy) is equal to ${1\over \kappa+1}$. Therefore, this bound is tight.
\qed\end{proof}
\medskip
\medskip

\begin{proof}{Proof of Theorem~\ref{thm:uniform}:}
This proof is similar to the proof of~\cite{KleinbergPR04} for submodular set functions. The main technical difficulty in proving our claim is that the optimal adaptive policy here is a random set whose distribution depends on the realized values of the elements of $\A$.

Let $P$ denote the (random) set chosen by an optimal adaptive policy.
Also, denote the marginal value of the $t$-th element chosen by the myopic policy by $\Delta_t$,
i.e., $\Delta_t = F(S_t) - F(S_{t-1}).$ Now consider a realization $s_t$ of $S_t$. Because $F$ is stochastic monotone submodular, we have
\begin{equation} \label{eq:uniform1}
\E[F(P)|s_{t}] \leq \E[F(P \cup S_t)|s_t] \le \E[F(S_{t}) + \sum_{l\in P } (F(S_{t} + l) - F(S_t)) |s_t].
\end{equation}

The above expectations  are taken over all realization of $P$ such that the realized values of elements in $S_t$ are according to $s_t$.
Because the myopic policy chooses the element with maximum marginal value, for every $l \in P$, we have
$\E[\Delta_{t+1}|s_t] \ge \E[F(S_{t} + l) - F(S_t)|s_t].$
Therefore, we get $\E[F(P)|s_{t}] \le \E[F(S_{t}) + k \Delta_{t+1}|s_t].$ Since this inequality  holds for every possible path in the history, by
adding up all such inequalities for all $t$, $0\le t\le k-1$, we have
\begin{eqnarray*}
\nonumber \E[F(P)] & \le & \E[F(S_{t})] + k \E[\Delta_{t+1}] 
= \E[\Delta_1 + \cdots + \Delta_t] + k \E[\Delta_{t+1}].
\end{eqnarray*}

We multiply the $t$-th inequality, $0\le t\le k-1$,  by $\kfac^{k-1-t}$, and add them all.
The sum of the coefficients of $\E[F(P)]$ is equal to
\begin{eqnarray}
\label{eq:ag_rhs1}
    \sum_{t=0}^{k-1} \kfac^{k-1-t} = \sum_{t=0}^{k-1} \kfac^t = \frac{1 - \kfac^{k}}{1-\kfac} = k(1 - \kfac^{k}).
\end{eqnarray}
On the right hand side, the sum of the coefficients corresponding to the term $\E[\Delta_t]$, $1\le t\le k$, is equal to
\begin{eqnarray}
 k \kfac^{k-t} + \sum_{j=t}^{k-1} \kfac^{k-1-j} & = &
\kfac^{k-t} + k(1 - \kfac^{k-t})  =  k \label{eq:ag_lhs1}.
\end{eqnarray}

Thus, by inequalities~(\ref{eq:ag_rhs1}) and~(\ref{eq:ag_lhs1}), we conclude
$ (1 - \kfac^k) \E[F(P)] \le  \sum_{t=1}^{k} \E[\Delta_t] = \E[F(S_k)].$
Hence, the approximation ratio of the myopic policy is at least $\efac$.
\qed\end{proof} 
\medskip

\medskip\medskip

%% file: online_appendix.tex
\centerline{\LARGE{Online Appendix}}

\section{Proofs}

\subsection{Proofs From Section~\ref{sec:tightexample}}

\begin{proof}{Proof of Lemma~\ref{lem:example_nonadaptive}}
Consider an arbitrary non-adaptive policy which picks set $S\subset\A$, containing $m^2$ sets from $\F$.  For each $i$, define $k_i = |S \cap \F^{(i)}|$.

Moreover, each element $i$ in the ground set is covered if and only if at least one of its corresponding chosen subsets are realized as a non-empty subset. Hence, it will be covered with probability $1 - (1 - \tfrac{1}{m})^{k_i}$. Therefore, the expected value of this policy is $\sum_i 1 - (1 - \tfrac{1}{m})^{k_i}$. Note that $1 - (1 - \tfrac{1}{m})^{x}$ is a concave function with respect to $x$ and also $\sum_i k_i = m^2$. Hence,  the expected value of the policy is maximized when $k_1 = k_2 = \cdots = k_m = m$. In this case, the expected value is $(1 - (1 - \tfrac{1}{m})^{m})m \approx (\efac)m$ for large $m$.
\qed\end{proof}

\begin{proof}{Proof of Lemma~\ref{lem:example_adaptive}}
Let $X_k$ be the indicator random variable corresponding to the event that the subset chosen at the $k$-th step is realized as a non-empty subset for any $1 \leq k \leq m^2$. Note that the number of elements covered by $\pi$ is $\sum_{k = 1}^{m^2} X_k$. Moreover, all $X_k$'s are independent random variables.

By the description of $\pi$, as long as $\sum_{i = 1}^k X_k < m^2$, $X_k$ will be one with probability $\tfrac{1}{m}$ and will be zero with probability $1 - \tfrac{1}{m}$. Also, when $\sum_{i = 1}^t X_k = m$, we have already covered all the elements in the ground set. Therefore, $X_{t + 1}, \cdots, X_{m^2}$ will all be equal to zero. With this observation, we define i.i.d random variables $Y_1, Y_2, \cdots, Y_{m^2}$, where each $Y_i$ is set to be one with probability $\tfrac{1}{m}$ and zero with probability $\tfrac{1}{m}$. Observe that $\min\{m, Y = \sum_k Y_k\}$ has the same probability distribution as $\sum_k X_k$.
Note that $\E[Y] = m$. Using Chernoff bound, we have
$\Pr[Y \leq m - m^{2/3}] \leq e^{-\frac{{m^{4/3}}}{2m}} = e^{-m^{1/3}}.$
Thus, with probability at least $1 - e^{-m^{1/3}}$, we have $Y > m - m^{2/3}$. Hence,
$\E[\sum_{k = 1}^{m^2} X_k] = \E[\min\{m, Y\}] \geq  (1 - e^{-m^{1/3}})(m - m^{2/3}) = m - o(m),$
which completes the proof of the lemma.
\qed\end{proof}

\subsection{Proofs From Section~\ref{sec:bounds_non_adapt}}

\begin{proof}{Proof of Lemma~\ref{lem:submodularity}}
The intuition behind the proof is that $F$ can be thought of as a linear combination of some monotone submodular functions. For monotonicity, let $S \subsetneq \A$ and $i$ be so that $X_i \notin S$. Consider any arbitrary realization $\bar{\Theta}_S = (\theta_1, \theta_2, ..., \theta_n)$, and let $Z=(\zeta_1, \zeta_2, \cdots, \zeta_n)$ be their corresponding values in Eq.~(\ref{eq:value}) such that $\ff(\bar{\Theta}_S) = f(Z)$. Because $X_i \notin S$, we have $\theta_i = \nill$ and consequently, $\zeta_i = 0$. Now, by adding $X_i$ to $S$ we will have the following.
\begin{eqnarray*}
E[\ff(\Theta_{S \cup \{X_i\}}) | \bar{\Theta}_S] &=& \int_{x_i \in \Omega_i} f(\zeta_1, \cdots, \zeta_{i - 1}, x_i, \zeta_{i+1}, \cdots, \zeta_n) g(x_i) dx_i \\
&\geq&\int_{x_i \in \Omega_i} f(\zeta_1, \cdots, \zeta_{i - 1}, 0, \zeta_{i+1}, \cdots, \zeta_n) g(x_i) dx_i \\
&=& f(\zeta_1, \cdots, \zeta_{i - 1}, 0, \zeta_{i+1}, \cdots, \zeta_n) =  \ff(\bar{\Theta}_S),
\end{eqnarray*}
where the inequality holds due to the monotonicity of the function $f$ and the fact that $\Omega_i \subseteq \R_+$.
Now, we apply the derived inequality to bound the value of $F$ over the subset $S \cup \{X_i\}$.
\begin{eqnarray*}
F(S \cup \{X_i\}) = \E[\ff(\Theta_{S\cup\{X_i\})}]&=& \int_{\bar{\Theta}_S \in \Omega} E[\ff(\Theta_{S\cup\{X_i\})}|\bar{\Theta}_S]g_S(\bar{\Theta}_S)d\bar{\Theta}_S\\
&\geq&\int_{\bar{\Theta}_S \in \Omega} \ff(\bar{\Theta}_S)g_S(\bar{\Theta}_S)d\bar{\Theta}_S = E[\ff(\Theta_S)] = F(S).
\end{eqnarray*}
This completes the proof of monotonicity of function $F$. 

The proof of submodularity also follows from a similar path. 
 Let $S$ and $T$ be any arbitrary subsets of the ground set $\A$. 
We have 
$$F(S) + F(T) = E[f(\Theta_S)] + E[f(\Theta_T)],$$
where $\Theta_S$ (resp. $\Theta_T$) is the realization of the elements if we choose the set of elements in $S$ (resp. $T$). Hence,
\begin{equation}
\label{eq:subalaki} F(S) + F(T) = \int_{\Theta: \theta_i = \nill, i \notin S} f(\Theta) \prod_{i: X_i \in S} g_i(\theta_i)d\theta_i
 + \int_{\Theta: \theta_i = \nill, i \notin T} f(\Theta) \prod_{i: X_i \in T} g_i(\theta_i)d\theta_i.
\end{equation}


Note that $g_i$'s are  probability distributions corresponding to independent random variables. Hence, for every set $A \in \A$ we have 
$$\int_{\theta_i \in \Omega_i: X_i \in A} \prod_{i: X_i \in A} g_i(\theta_i) d\theta_i = 1.$$

Combining the above equality for $A = \A \setminus S$ and $A = \A \setminus T$ with Eq.~(\ref{eq:subalaki}) results in the following.
\begin{eqnarray*}
 F(S) + F(T) = &&\int_{\Theta: \theta_i = \nill, i \notin S} f(\Theta) \left( \int_{\theta_i \in \Omega_i: X_i \notin S} \prod_{i: X_i \notin S} g_i(\theta_i) d\theta_i \right) \prod_{i: X_i \in S} g_i(\theta_i)d\theta_i \\
 &+& \int_{\Theta: \theta_i = \nill, i \notin T} f(\Theta) \left( \int_{\theta_i \in \Omega_i: X_i \notin T} \prod_{i: X_i \notin T} g_i(\theta_i) d\theta_i \right) \prod_{i: X_i \in T} g_i(\theta_i)d\theta_i \\
 = && \int_{\Theta} f(\Theta(S)) \prod_i g_i(\theta_i) d\theta_i + \int_{\Theta} f(\Theta(T)) \prod_i g_i(\theta_i) d\theta_i \\
 = && \int_{\Theta} \Big(f(\Theta(S)) + f(\Theta(T))\Big) \prod_i g_i(\theta_i) d\theta_i,
\end{eqnarray*}

where for every $A \in \A$, $\Theta(A)$ is defined as $(\theta_1\times \1_{X_1 \in A}, \theta_2 \times \1_{X_2 \in A}, \cdots, \theta_n\times \1_{X_n \in A})$.

The submodularity of means that 
$$f(\Theta(S)) + f(\Theta(S)) \geq f(\Theta(S) \vee \Theta(T)) + f(\Theta(S) \wedge \Theta(T)).$$
However, $\Theta(S) \vee \Theta(T)$ is nothing but the component-wise maximum of $\Theta(S)$ and $\Theta(T)$. Hence, its $i$-th entry is the pairwise maximum of $\theta_i\times \1_{X_i \in S}$ and $\theta_i\times \1_{X_i \in T}$, or equivalently, $\theta_i\times \1_{X_i \in S \cup T}$. Similarly, the $i$-th entry of 
$\Theta(S) \wedge \Theta(T)$ is  $\theta_i\times \1_{X_i \in S \cap T}$. 

In summary,
\begin{eqnarray*}
 F(S) + F(T) &=& \int_{\Theta} \Big(f(\Theta(S)) + f(\Theta(T))\Big) \prod_i g_i(\theta_i) d\theta_i \\
    &\geq& \int_{\Theta} \Big(f(\Theta(S) \vee \Theta(T)) + f(\Theta(S) \wedge \Theta(T))\Big) \prod_i g_i(\theta_i) d\theta_i \\
    &\geq& \int_{\Theta} \Big(f(\Theta(S \cup T)) + f(\Theta(S \cap T))\Big) \prod_i g_i(\theta_i) d\theta_i \\
    &=& \E[f(\Theta_{S \cup T})] + \E[f(\Theta_{S \cap T})] \\
    &=& F({S \cup T}) + F({S \cap T}).
\end{eqnarray*}

This completes the proof of the lemma.
\qed\end{proof}

\subsection{Proof from Section~\ref{sec:polytight}} \label{apx:polyproofs}
\begin{proof}{Proof of Lemma~\ref{superkhafan}} By definition~(\ref{eq:magic}), the $j$-th entry of $\Theta_{R(t)}^{\uparrow j}$, for any fixed $R(t)$ is as follows:
\begin{itemize}
\item If $X_j \in R(t)$, then it is the maximum of two independently sampled random variables from the c.d.f. $G_j$. Hence, its own c.d.f. will be $G_j^2$.
\item If $X_j \notin R(t)$, then it is simply a random variable sampled from c.d.f. $G_j$. 
\end{itemize}
Similar to the rest of the paper, we assume that we have oracle access to such values, i.e. we know the values $\E_{\Theta}[f(\Theta_{R(t)})]$ and $\E_{\Theta}[f(\Theta_{R(t)}^{\uparrow j})]$ (and consequently, the value of $\E_{\Theta}\left[f(\Theta_{R(t)}^{\uparrow j}) - f(\Theta_{R(t)})\right]$) when $R(t)$ is fixed and the expectation is taken over all realizations of $\Theta$. (For computing these values using sampling refer to the appendix.) For the simplicity of notation, for a fixed $R(t)$ we define $f(R(t), j) = \E_{\Theta}\left[f(\Theta_{R(t)}^{\uparrow j}) - f(\Theta_{R(t)})\right]$. The interpretation of $f(R(t), j)$ is the expected marginal contribution of adding element $j$ to the set $R(t)$, while allowing to improve the value of $\theta_j$ if $j$ is already in $R(t)$.

Note that similar to Lemma 3.2 in~\cite{CalinescuCPV08} $R(t)$ is a random set containing each element $j$ independently at random with probability $y_j$. Now, consider an estimate $w_j(t)$ of $\E_{R}\left[f(R(t), j)\right]$ obtained by blah independent samples $R_i$ of $R(t)$.

Let us call an estimate $w_j$ \emph{bad} if $|w_j(t) - \E_R[f(R(t), j)]| > \delta \OPT$. Following the proof of that lemma, we define $X_i = (f(R_i, j) - \E_R[f(R(t), j)])/\OPT$, $k = \frac{4}{\delta^2}(1+\ln n - 0.5\ln \delta)$, and $a = \delta k$. We have $|X_i| \leq 1$ since $\OPT \geq \max_j f(\{\}, j) \geq \max_{R \subseteq \A} f(R, j)$, where the first inequality holds because of monotonicity of $f$ and the second inequality is due to the submodularity of $f$ and hence, diminishing marginal return of any element $j$. The estimate is bad if and only if $|\sum_i X_i| > a$. But the Chernoff bound (see Theorem 2.2 in~\citealt{CalinescuCPV08}) implies that $\Pr[|\sum_i X_i| > a] \leq 2e^{-\delta^2k/2} = 2e^{-2-2\ln n + \ln \delta} \leq \delta/(3n^2)$.

Note that in each step we compute $n$ estimates (one for each $X_j \in \A$) and the total number of steps is $1/\delta$. By the union bound, the probability of having any bad estimate is at most $\frac{n}{\delta}\times \frac{\delta}{3n^2}$. Hence, with probability at least $(1 - \tfrac{1}{3n})$ all estimates throughout the algorithm are good, i.e., 
\begin{equation}
|w_j(t) - \E_R[f(R(t), j)]| \leq \delta\OPT,
\end{equation}
for all $j$ and $t$. 
Now, let $I \in \I$ be any independent set. Note that $|I| = d$, and with high probability there is no bad estimates. Thus, with high probability
\begin{equation}\label{eq:dobaande} \left(\sum_{j: X_j \in I} \E_R[f(R(t), j)]\right) + d\delta~.~\OPT \geq \sum_{j: X_j \in I} w_j(t) \geq \left(\sum_{j: X_j \in I} \E_R[f(R(t), j)]\right) - d\delta~.~\OPT.
\end{equation}

Now, let $I^*$ denote an independent set achieving the  $\max_{I \in \I} \sum_{j: X_j \in I} \E_{\Theta, R}\left[f(\Theta_{R(t)}^{\uparrow j}) - f(\Theta_{R(t)})\right]$. Our algorithm finds a set $I(t) \in \I$ that maximizes $\sum_{j: X_j \in I} w_j$. Hence, with high probability
\begin{eqnarray*}
\sum_{j: X_j \in I(t)} \E_{\Theta, R}\left[f(\Theta_{R(t)}^{\uparrow j}) - f(\Theta_{R(t)})\right] &=&
\sum_{j: X_j \in I(t)} \E_R\left[\E_{\Theta}\left[f(\Theta_{R(t)}^{\uparrow j}) - f(\Theta_{R(t)})\right]\right] \\
\framebox[1.05\width]{by definition of $f(R(t), j)$}~~~~~~~~~~ &=& \sum_{j: X_j \in I(t)}  \E_R[f(R(t), j)] \\
\framebox[1.05\width]{by the first inequality in~(\ref{eq:dobaande})}~~~~~~~~~~ &\geq& \sum_{j: X_j \in I(t)} w_j(t)  - d\delta~.~\OPT \\
\framebox[1.05\width]{by the choice of $I(t)$}~~~~~~~~~~ &\geq& \sum_{j: X_j \in I^*} w_j(t)  - d\delta~.~\OPT \\
\framebox[1.05\width]{by the second inequality in~(\ref{eq:dobaande})}~~~~~~~~~~ &\geq& \sum_{j: X_j \in I^*}  \E_R[f(R(t), j)] - 2d\delta~.~\OPT \\
\framebox[1.05\width]{by definition of $f(R(t), j)$}~~~~~~~~~~ &=& \sum_{j: X_j \in I^*} \E_{\Theta, R}\left[f(\Theta_{R(t)}^{\uparrow j}) - f(\Theta_{R(t)})\right]  -  2d\delta~.~\OPT,
\end{eqnarray*}
which completes the proof of the lemma.
\qed\end{proof}
\medskip

\begin{proof}{Proof of Lemma~\ref{updatebound}:}
 As before, let $R(t)$ denote the random set that contains $X_j$ with probability $y_j(t)$ independently at random. Also, let $D(t)$ be a random set that contains $X_j$ with probability $\delta_j(t) = y_j(t+\delta) - y_j(t)$ independently at random. Note that according to our algorithm $\delta_j(t) = \delta~.~\1_{j \in I(t)}$. We show that $F(y(t+\delta)) \geq \E_{\Theta, R, D}[f(\Theta_{R(t)} \vee \Theta_{D(t)})]$, where the expectation is taken over all the random sets $R(t)$ and $D(t)$ and their corresponding realization of $\Theta$. We emphasize that due to the independence of each entry and also the monotonicity of $f$, we only need to show the entry-by-entry stochastic dominance of the left hand side over the right hand side. Let $\alpha(x)$ be the c.d.f. of the $j$-th entry of the left hand side, i.e. $F(y(t+\delta))$. Also, let $\beta(x)$ denote the c.d.f. of the $j$-th entry of $\E_{\Theta, R, D}[f(\Theta_{R(t)} \vee \Theta_{D(t)})]$. By the definition of $F$ we know that $F(y(t+\delta)) = \E_{\Theta, R}[f(\Theta_{R(t+\delta)})]$, where each $X_j$ is contained in $R(t+\delta)$ with probability $y_j(t+\delta) = y_j(t) + \delta_j(t)$. Hence, for the c.d.f. of its $j$-th entry we have 
$$\alpha(x) = 1 - \Big(y_j(t) + \delta_j(t)\Big) + G_j(x)\Big(y_j(t) + \delta_j(t)\Big),$$
for every $x \in \R_+$. 

On the other hand, the $j$-th entry of $\Theta_{R(t)} \vee \Theta_{D(t)}$ can be understood  as follows. The value of the entry is $0$ iff $X_j \notin R(t)$ and $X_j \notin D(t)$. This happens with probability $(1 - y_j(t))(1 - \delta_j(t))$. Similarly, with probability $y_j(t)(1 - \delta_j(t)) + (1 - y_j(t))\delta_j(t)$ it will be a variable drawn from the c.d.f. $G_j$. Finally, with probability $y_j(t)\delta_j(t)$ it will be the maximum of two independent random variables each drawn from c.d.f. $G_j$, where consequently, its c.d.f. will be $G_j^2$. In conclusion, the c.d.f. will be 
$$\beta(x) = \bigg((1 - y_j(t))(1 - \delta_j(t))\bigg) + \bigg(y_j(t)(1 - \delta_j(t)) + (1 - y_j(t))\delta_j(t)\bigg) G_j(x) + \bigg(y_j(t)\delta_j(t)\bigg)G_j(x)^2,$$
for every $x \in R_+$. But we have,
\begin{eqnarray*}
\beta(x) - \alpha(x) &=& \Big(y_j(t)\delta_j(t)\Big) - \Big(2y_j(t)\delta_j(t) \Big)G_j(x) + \Big(y_j(t)\delta_j(t)\Big) G_j(x)^2 \\
&=& \Big(y_j(t)\delta_j(t)\Big)\Big(1-G_j(x)\Big)^2 \\
&\geq& 0.
\end{eqnarray*}
This establishes the entry-by-entry stochastic dominance of 
$F(y(t+\delta)) = \E_{\Theta, R}[f(\Theta_{R(t+\delta)})]$ over $\E_{\Theta, R, D}[f(\Theta_{R(t)} \vee \Theta_{D(t)})]$. Thus, 
\begin{eqnarray*}
F(y(t+\delta)) - F(y) &=& \E_{\Theta, R}[f(\Theta_{R(t+\delta)})] - \E_{\Theta, R}[f(\Theta_{R(t)})] \\
&\geq& \E_{\Theta, R, D}[f(\Theta_{R(t)} \vee \Theta_{D(t)})] - \E_{\Theta, R}[f(\Theta_{R(t)})] \\
&\geq& \sum_{j} \Pr[D(t) = \{X_j\}] \bigg(\E_{\Theta, R}[f(\Theta_{R(t)} \vee \Theta_{\{X_j\}})] - \E_{\Theta, R}[f(\Theta_{R(t)})]\bigg) \\
\framebox[1.05\width]{by the definition of $\Theta_{R(t)}^{\uparrow j}$}~~~~~ &=& \sum_{j: X_j \in I(t)} \delta(1-\delta)^{|I(t)|-1}\E_{\Theta, R}\left[f(\Theta_{R(t)}^{\uparrow j}) - f(\Theta_{R(t)})\right] \\
\framebox[1.05\width]{ by $|I(t)| = d$~}~~~~~&\geq& \delta(1-d\delta)\sum_{j: X_j \in I(t)} \E_{\Theta, R}\left[f(\Theta_{R(t)}^{\uparrow j}) - f(\Theta_{R(t)})\right].
\end{eqnarray*}
\qed\end{proof}
\medskip

\begin{proof}{Proof of Lemma~\ref{lem:final}:}
Let $\bary \in B(\M)$ be an arbitrary point in the base polytope of $\M$.  This point can be written as the convex combination of some independent sets of $\M$.
By taking this convex combination on the corresponding inequalities proved in Lemma~\ref{superkhafan} with high probability we have
\begin{eqnarray*}
\sum_{j: X_j \in I(t)} &&\E_{\Theta, R}\left[f(\Theta_{R(t)}^{\uparrow j}) - f(\Theta_{R(t)})\right] \geq \max_{I \in \I} \sum_{j: X_j \in I} \E_{\Theta, R}\left[f(\Theta_{R(t)}^{\uparrow j}) - f(\Theta_{R(t)})\right] - 2d \delta\OPT \\
&\geq& \sum_{j \in \A} \bary_j\E_{\Theta, R}\left[f(\Theta_{R(t)}^{\uparrow j}) - f(\Theta_{R(t)})\right] - 2d \delta\OPT \\
&=& - F(y(t)) + {\left[\E_{\Theta, R}[f(\Theta_{R(t)})] +  \sum_{j \in \A} \bary_j\E_{\Theta, R}\left[f(\Theta_{R(t)}^{\uparrow j}) - f(\Theta_{R(t)})\right]\right]}  - 2d \delta\OPT \\
&\geq& - F(y(t)) + {\inf_{\Theta}\left[f(\Theta) +  \sum_{j \in \A} \bary_j\left[f(\Theta^{\uparrow j}) - f(\Theta)\right]\right]}  - 2d \delta\OPT,
\end{eqnarray*}
where the first equality is due to the fact that  by definition $F(y(t)) = \E_{\Theta, R}[f(\Theta_{R(t)}]$.

However, the infimum part is nothing but the definition of $f^*(\bary)$ in~(\ref{eq:defstar}). Therefore, with high probability
$$\sum_{j: X_j \in I(t)} \E_{\Theta, R}\left[f(\Theta_{R(t)}^{\uparrow j}) - f(\Theta_{R(t)})\right] \geq f^*(\bary) - F(y(t)) - 2d\delta~.~\OPT,$$
for any arbitrary $\bary \in B(\M)$. Hence, with high probability
$$\sum_{j: X_j \in I(t)} \E_{\Theta, R}\left[f(\Theta_{R(t)}^{\uparrow j}) - f(\Theta_{R(t)})\right] \geq \sup_{y \in B(\M)}\{f^*(y)\} - F(y(t)) - 2d\delta~.~\OPT.$$
Also, note that $\OPT$ is the optimum value of the adaptive policy which is bounded by $\sup_{y \in B(\M)} f^+(y)$ (see Lemma~\ref{adaptiveupperbound}). Also, we know from inequality~(\ref{starplus}) that $f^+(y) \leq f^*(y)$ for every $y \in B(\M)$. Hence, $\OPT \leq \sup_{y \in B(\M)} f^*(y)$. Consequently, with high probability
$$\sum_{j: X_j \in I(t)} \E_{\Theta, R}\left[f(\Theta_{R(t)}^{\uparrow j}) - f(\Theta_{R(t)})\right] \geq (1 -  2d\delta) \sup_{y \in B(\M)}\{f^*(y)\} - F(y(t)).$$
By applying the result of Lemma~\ref{updatebound} on the above inequality, with high probability we get the following.
\begin{eqnarray*} 
F(y(t + \delta)) - F(y(t)) &\geq& \delta (1- d\delta) \left[ (1 -  2d\delta) \sup_{y \in B(\M)}\{f^*(y)\} - F(y(t))\right] \\
&\geq& \delta \left[(1 -  3d\delta) \sup_{y \in B(\M)}\{f^*(y)\} - F(y(t))\right],
\end{eqnarray*}
which completes the proof of the lemma.
\qed\end{proof}
\medskip

%



\subsection{Proof from Section~\ref{sec:f_computation}}  \label{apx:f_computation}
\begin{proof}{Proof of Proposition~\ref{lem:sampling}}
Note that $F(S) = \E[\ff(\Theta_S)]$. By Lemma~\ref{lem:lipbound} below we have $\Var[\ff(\Theta_S)] \leq K^2 |S|^2 \V$. Suppose that we take $t$ independent samples of the value of $\ff(\Theta_S)$ and take their average as an estimation for the actual value of $\ff(\Theta_S)$. The derived sample will have a variance lower than $K^2 |S|^2 \V/t$. The proof is completed by taking $t = \lceil K^2 n^2 \V \epsilon^{-1} \delta^{-0.5} \rceil$ and applying the Chebyshev inequality.
\qed\end{proof}
\medskip

\begin{proof}{Proof of Proposition~\ref{lem:sampling2}} We use Chernoff bound. In particular, we use Theorems 6 and 7 in~\cite{fanchung}. Clearly, since $0 \leq f(\Theta_S) \leq \FF$, we have $\Var[f(\Theta_S)] \leq \FF^2/2$. Suppose we take $t$ samples and let $\Upsilon_t$ be their total summation. Theorem 6 in~\cite{fanchung} ensures that 
$$\Pr[\Upsilon_t \geq t.f(\Theta_S) + \lambda] \leq e^{-\tfrac{\lambda^2}{2(\Var(\Upsilon_t)+\FF\lambda/3)}}.$$
Note that $\Upsilon_t$ consists of $t$ independent samples of $f(\Theta_S)$. It means that $\Var(\Upsilon_t) = t.\Var(f(\Theta_S))$, and hence, it is bounded from above by $t\FF^2/2$. Now, let $\lambda = t\epsilon $. We will have
$$\Pr\left[\frac{\Upsilon_t}{t} \geq f(\Theta_S) (1+\epsilon)\right] \leq \textrm{exp}\Big(-\frac{t^2 \epsilon^2 }{2(t\FF^2/2+\FF t \epsilon f(\Theta_S)/3)}\Big) \leq \textrm{exp}\Big(-\frac{t^2 \epsilon^2}{2t\FF^2}\Big) = \textrm{exp}\Big(-\frac{t\epsilon}{2\FF^2}\Big).$$

We emphasize that our bound in the second inequality is loose. This is only for the sake of achieving the same upper-tail and the lower-tail bounds and subsequently, having a clearer representation.

Now, Theorem 7 in~\cite{fanchung} implies that
$$\Pr[\Upsilon_t \leq t.f(\Theta_S) - \lambda] \leq e^{-\tfrac{\lambda^2}{2t\FF^2}}.$$
Again, by letting $\lambda = t\epsilon $ we will have
$$\Pr\left[\frac{\Upsilon_t}{t} \geq f(\Theta_S) (1+\epsilon)\right] \leq \textrm{exp}\Big(-\frac{t^2\epsilon^2}{2t\FF^2}\Big)= \textrm{exp}\Big(-\frac{t\epsilon}{2\FF^2}\Big).$$

Now, we can conclude the following for the chance that the average of $t$ samples (i.e. $\Upsilon_t/t$) deviates from its expected value (i.e. $f(\Theta_S)$) by a factor more than $(1\pm \epsilon)$.
$$\Pr\left[f(\Theta_S) (1-\epsilon) \leq \frac{\Upsilon_t}{t} \leq f(\Theta_S) (1+\epsilon)\right] \geq  1-2\textrm{exp}\Big(-\frac{t\epsilon}{2\FF^2}\Big).$$
The above inequality for $t = \lceil-2\FF^2 \ln(\delta/2) \epsilon^{-1}\rceil$ completes the proof. 
\qed\end{proof}
\medskip

\subsubsection{Bounding the Variance of $\ff(\Theta_S)$:} 
Suppose $S = \{X_{i_1}, X_{i_2}, \cdots, X_{i_m}\}$ and let $S_j =  \{X_{i_1}, X_{i_2}, \cdots, X_{i_j}\}$. We write $F(S)$ as the telescopic sum of its marginal values.  
$$\ff(\Theta_S) = \sum_{j = 1}^m [\ff(\Theta_{S_j}) - \ff(\Theta_{S_{j-1}})].$$
First, we bound the variance of each term of the summation.
\begin{lemma}
\label{lem:lip}
For every $1 \leq j \leq m$,  if $f$ is $K$-Lipschitz continuous, then $$\Var[\ff(\Theta_{S_j}) - \ff(\Theta_{S_{j-1}})] \leq K^2 \Var[X_{i_j}].$$
\end{lemma}
\begin{proof}{Proof}
We will prove a stronger result here. We show that the claim holds for any possible realizations of $\Theta_{S_{j-1}}$. Fix $\Theta_{S_{j-1}} = (\theta_1, \theta_2, \cdots, \theta_n)$. Note that $\Theta_{S_{j}}$ will be different from $\Theta_{S_{j-1}}$ in only one dimension, namely the one associated with $X_{i_j}$. Define $h(x) := \ff(\Theta_{S_j}) - \ff(\Theta_{S_{j-1}})$ for the specific realization in which the value of $X_{i_j}$ is $x$. Equivalently,
\begin{eqnarray*}
&& h(x) = f(\zeta_1, \cdots, \zeta_{i_{j-1}}, x, \zeta_{i_{j+1}}, \cdots, \zeta_n) - f(\zeta_1, \cdots, \zeta_n), ~~~~~~\forall x \in \R_+\\
&& \textrm{where~}
\zeta_i = \left\{\begin{array}{ll} \theta_i & \theta_i = x_i, \\ 0 & \theta_i = \nill. \end{array}\right.
\end{eqnarray*}
Note that $X_{i_j} \notin S_{i_{j-1}}$, hence, $\theta_{i_j} = \nill$ and consequently, $\zeta_{i_j} = 0$. Therefore, because $f$ is monotone, the function $h$ defined above will be non-negative and increasing. Moreover, $h$ is also $K$-Lipschitz continuous.
%
%
%
Let $\mu = \E[X_{i_j}]$. Due to the $K$-Lipschitz continuity of $h(.)$, we have $h(x) \leq h(\mu) + K(x-\mu)$ for $x > \mu$. Also, we have $h(x) \geq h(\mu) - K(\mu-x)$ for $x<\mu$.
Hence, the variance of $h$ can be written as follows.
\begin{eqnarray*}
\Var[h(X_{i_j})] = \E\left[\left(h(X_{i_j}) - \E[h(X_{i_j})]\right)^2\right] 
&\leq& \E\left[\left(h(X_{i_j}) - h(\mu)\right)^2\right] \\
&\leq& \int_{x \in \Omega_{i_j}} K^2(x - \mu)^2  g_{i_j}(x) dx = K^2 \Var[X_{i_j}].
\end{eqnarray*}
Note that the second inequality  holds because  for every random variable $X$, the function $\nu(c) = \E[(X-c)^2]$ is minimized at $c = \E[X]$. This completes the proof of lemma.
\qed\end{proof}
\medskip
Now, we are ready to bound the variance of $\ff(\Theta_S)$.
\begin{lemma}
\label{lem:lipbound}
If $f$ is $K$-Lipschitz continuous, and $\V = max_{X_i \in S} \Var[X_i]$, then $\Var[\ff(\Theta_S)] \leq |S|^2K^2\V$.
\end{lemma}
\begin{proof}{Proof}
Define $Y_j = \ff(\Theta_{S_j}) - \ff(\Theta_{S_{j-1}})$. Hence, we have $\ff(\Theta_S) = \sum_{j = 1}^m Y_j$. By Lemma~\ref{lem:lip}, for all $j$, $\Var[Y_j] \leq K^2\V$. We have,
\begin{eqnarray*}
\Var[\ff(\Theta_S)] &=& \sum_{j: X_j \in S} \Var[Y_j] + \sum_{i < j: \{X_i, X_j\} \subset S} 2\Cov[Y_i, Y_j] \\
&\leq& |S| K^2\V + \sum_{i < j: \{X_i, X_j\} \subset S} 2\sqrt{K^2\Var[Y_i].K^2\Var[Y_j]} \\
&\leq& |S|K^2 \V + |S|(|S|-1)K^2 \V = |S|^2K^2\V.
\end{eqnarray*}
\qed\end{proof}

\upt{
\subsection{Proofs from Section~\ref{sec:mult}}
\begin{proof}{Proof of Lemma~\ref{lem:replacement}}
If $f$ is monotone, then it is easy to see that $\phi$ is also monotone. We now prove that $\phi$ is submodular.
Fix $x, y  \in \R_+^{mn}$. Let 
\begin{eqnarray*}
x &=& (x^{(1)}_1, \cdots,  x^{(1)}_{m}, x^{(2)}_1, \cdots,  x^{(2)}_{m}, \cdots, x^{(n)}_1, \cdots,  x^{(n)}_{m}), \textrm{ and} \\
y &=& (y^{(1)}_1, \cdots,  y^{(1)}_{m}, y^{(2)}_1, \cdots,  y^{(2)}_{m}, \cdots, y^{(n)}_1, \cdots,  y^{(n)}_{m}).
\end{eqnarray*}
We will show that $\phi(x \vee y) + \phi(x \wedge y) \leq \phi(x) + \phi(y)$. For every $i$, define $\bar{x}^{(i)} := \max_k \{x^{(i)}_k\}$ and $\bar{y}^{(i)} := \max_k \{y^{(i)}_k\}$, and let $\bar{x} := (\bar{x}^{(1)}, \bar{x}^{(2)}, \cdots, \bar{x}^{(n)})$ and $\bar{y} := (\bar{y}^{(1)}, \bar{y}^{(2)}, \cdots, \bar{y}^{(n)})$. Note that by the definition of $\phi$ in Eq.~(\ref{eq:replacement}), we have $f(x) = \phi(\bar{x})$ and $f(y) = \phi(\bar{y})$. Also, we have the following equation for $\phi(x \vee y)$.
\begin{equation*}
\phi(x \vee y) = f\left(\max_k\left\{\max \{x^{(1)}_k, y^{(1)}_k\}\right\}, \max_k \left\{\max \{x^{(2)}_k, y^{(2)}_k\}\right\}, {\cdots}, \max_k \left\{\max \{x^{(n)}_k, y^{(n)}_k\}\right\}\right). \label{eq:vee}
\end{equation*}
Note tha, for every $i$ we have $\max_k \left\{\max \left\{x^{(i)}_k, y^{(i)}_k\right\}\right\} = \max\left\{ \max_k\left\{x^{(i)}_k\right\}, \max_k\left\{ y^{(i)}_k\right\}\right\}$. 
Therefore, we get $\phi(x\vee y) = f(\bar{x} \vee \bar{y})$. \newline
On the other hand, again by the definition of $\phi$ we have the following.
\begin{equation*}
\phi(x \wedge y) = f\left(\max_k\left\{\min \{x^{(1)}_k, y^{(1)}_k\}\right\}, \max_k \left\{\min \{x^{(2)}_k, y^{(2)}_k\}\right\}, {\cdots}, \max_k \left\{\min \{x^{(n)}_k, y^{(n)}_k\}\right\}\right). 
\end{equation*}
However, the infamous max-min inequality  implies that for every $i$ we have $$\max_k \left\{ \min \left\{x^{(i)}_k, y^{(i)}_k\right\} \right\} \leq \min \left\{\max_k \left\{x^{(i)}_k\right\}, \max_k \left\{y^{(i)}_k\right\}\right\} = \min\left\{\bar{x}^{(i)}, \bar{y}^{(i)}\right\}.$$ 
Hence, using the monotonicity of $f$, we get the following inequality.
\begin{eqnarray*}
\nonumber \phi(x \wedge y) &\leq& f\left(\min\{\bar{x}^{(1)}, \bar{y}^{(1)}\}, \min\{\bar{x}^{(2)}, \bar{y}^{(2)}\}, \cdots, \min\{\bar{x}^{(n)}, \bar{y}^{(n)}\}\right) \label{eq:wedge} = f(\bar{x} \wedge \bar{y}).
\end{eqnarray*}
Therefore,
\begin{eqnarray*}
\phi(x \vee y) + \phi(x \wedge y) &\leq& f(\bar{x} \vee \bar{y}) + f(\bar{x} \wedge \bar{y}) \le f(\bar{x}) + f(\bar{y}) = \phi(x) + \phi(y),
\end{eqnarray*}
where the second inequality holds because $f$ is submodular. This completes the proof of the lemma.
\qed\end{proof}
\medskip
}

\upt{\subsection{Proofs from Section~\ref{apx:nonmonotone}}
\begin{proof}{Proof of Lemma~\ref{lem:nejatbakhsh}}
Let us denote by $\psi:2^{\{1, 2, \cdots, n\}} \rightarrow \R$ be the submodular \emph{set} function defined by $\psi(S) = |\cup_{i\in S} A_i|$. Also,  Let $c:2^{\{1, 2, \cdots, n\}} \rightarrow \R$ be the \emph{additive} cost function defined by $c(S) := \sum_{i \in S} c_i$.
Consider $\theta, \eta$ be any two arbitrary vectors in $\{-1, 0, 1\}^n$. Define $S_\theta := \{i: \theta_i = 1\}$ and $S'_\theta := \{i: \theta_i = -1\}$. Also, let $T_\eta := \{i: \eta_i = 1\}$ and $T'_\eta := \{i: \eta_i = -1\}$. By the definition of our specific value function, we have 
\begin{eqnarray}
\nonumber f(\theta) &=& \psi(S_\theta) - c(S_\theta \cup S'_\theta) \textrm{~~~and,} \\
\label{eq:khasteshodam1} f(\eta) &=& \psi(T_\eta) - c(T_\eta \cup T'_\eta).
\end{eqnarray}
Now, we consider $\theta \vee \eta$ and $\theta \wedge \eta$. Since, $\vee$ and $\wedge$ denote the componentwise maximum and minimum, we have $\{i: (\theta \vee \eta)_i = 1\} = S_\theta \cup T_\eta$ and $\{i: (\theta \vee \eta)_i = -1\} = S'_\theta \cap T'_\eta$. Moreover, $\{i: (\theta \wedge \eta)_i = 1\} = S_\theta \cap T_\eta$ and $\{i: (\theta \wedge \eta)_i = -1\} = S'_\theta \cup T'_\eta$. Hence,
\begin{eqnarray}
\nonumber
f(\theta \vee \eta) &=& \psi(S_\theta \cup T_\eta) - c\big((S_\theta \cup T_\eta) \cup (S'_\theta \cap T'_\eta)\big) \textrm{~~~and,} \\
\label{eq:khasteshodam2} f(\theta \wedge \eta) &=& \psi(S_\theta \cap T_\eta) - c\big((S_\theta \cap T_\eta) \cup (S'_\theta \cup T'_\eta)\big).
\end{eqnarray}
Note that $S_\theta$ and $S'_\theta$ are disjoint. Also, $T_\eta$ and $T'_\eta$ are disjoint. Thus $S_\theta \cup T_\eta$ and $S'_\theta \cap T'_\eta$ are disjoint. Similarly, $S_\theta \cap T_\eta$ and $S'_\theta \cup T'_\eta$ are disjoint. Therefore, by combining Eq.~(\ref{eq:khasteshodam1}) and~(\ref{eq:khasteshodam2}) we will have the following. 
\begin{eqnarray*}
f(\theta \vee \eta) + f(\theta \wedge \eta) &=&  \psi(S_\theta \cup T_\eta) - \Big(c(S_\theta \cup T_\eta) + c(S'_\theta \cap T'_\eta)\Big) + \psi(S_\theta \cap T_\eta) - \Big(c(S_\theta \cap T_\eta) + (S'_\theta \cup T'_\eta)\Big) \\ 
\framebox[1.05\width]{by additivity of $c$}~~~ &=& \psi(S_\theta \cup T_\eta) +  \psi(S_\theta \cap T_\eta) - \Big(c(S_\theta \cup S'_\theta) + c(T_\eta \cup T'_\eta)\Big)\\
\framebox[1.05\width]{by submodularity of $\psi$}~~&\leq& \psi(S_\theta)+\psi(T_\eta) - \Big(c(S_\theta \cup S'_\theta) + c(T_\eta \cup T'_\eta)\Big) \\
&=& f(\theta) + f(\eta).
\end{eqnarray*}
This completes the proof of submodularity of $f$.
\qed\end{proof} \medskip
\begin{proof}{Proof of Lemma~\ref{lem:nonadaptivenonmonotone}}
A non-adaptive policy picks a subset $S \subseteq \{1, 2, \cdots, n\}$. Due to the symmetry between all $X_i$-s, the value of the policy  depends only on $|S|$. Let $i := |S|$. The expected value of the policy will be $\psi(i) := \left(1- (1-p)^i\right)-ic$ that is a continuous function in $i \in \R$. However, $\partial \psi(i)/ \partial i = - (1-p)^i \ln(1-p) - c$, which is positive for $i < 1$ and negative for $i>1$ when $c = -(1-p)\ln (1-p)$. Thus, $\psi(i)$ attains its maximum at $1$. Therefore, the optimum non-adaptive policy picks only one element and its expected value is $p - c$.
\qed\end{proof}\medskip
\begin{proof}{Proof of Lemma~\ref{lem:adaptivenonmonotone}}
By following policy $\pi$, the probability that $X_i$ turns out to be the first non-empty set is $(1-p)^{i - 1} p$. In that case, the value obtained by the policy is $1 - ic$. Hence, the expected value of $\pi$ is 
\begin{eqnarray*}
\sum_{i = 1}^n (1-p)^{i - 1} p (1 - ic) & = & \left(\sum_{i = 1}^n (1-p)^{i - 1} p\right) -  \left(\sum_{i = 1}^n (1-p)^{i - 1} p i\right)c \\
&\geq& \left(\sum_{i = 1}^n (1-p)^{i - 1} p\right) -  \left(\sum_{i = 1}^\infty (1-p)^{i - 1} p i\right)c \\
&=&  \left(1 - (1-p)^n\right)  - c/p,
\end{eqnarray*}
where the equality is obtained by using the formula for the expected value of the geometric distribution.
\qed\end{proof}\medskip
}